\documentclass[10pt,
]{amsart}
\usepackage{amssymb}
\newcommand{\no}[1]{#1}

\renewcommand{\no}[1]{}  \newcommand{\upDelta}{\Delta} 

\no{\usepackage{times}\usepackage[
subscriptcorrection, slantedGreek, 
nofontinfo]{mtpro}
\renewcommand{\Delta}{\upDelta}
}

\date{\today}

\numberwithin{equation}{section}%

\newtheorem{theorem}{Theorem}[section]
\newtheorem{proposition}{Proposition}[section]
\newtheorem{lemma}{Lemma}[section]
\newtheorem{definition}{Definition}[section]
\newtheorem{corollary}{Corollary}[section]

\theoremstyle{definition}
\newtheorem{remark}{Remark}[section]

\DeclareMathOperator{\Vol}{Vol}

\DeclareMathOperator{\grad}{grad}
\DeclareMathOperator{\supp}{supp}
\DeclareMathOperator{\Ker}{Ker}
\DeclareMathOperator{\Coker}{Coker}

\DeclareMathOperator{\WF}{WF}

\newcommand{\eps}{\varepsilon}

\newcommand{\R}{{\bf R}}
\newcommand{\Id}{\mbox{Id}}
\renewcommand{\r}[1]{(\ref{#1})}
\newcommand{\PDO}{$\Psi$DO}
\newcommand{\be}[1]{\begin{equation}\label{#1}}
\newcommand{\ee}{\end{equation}}

\renewcommand{\d}{\mathrm{d}}

\renewcommand{\i}{\mathrm{i}}

\newcommand{\bo}{\partial M}

\newcommand{\F}{F}
\newcommand{\B}{B}

\title[The geodesic X-ray transform with fold caustics]{The geodesic X-ray transform with fold caustics}

\author[P. Stefanov]{Plamen Stefanov}
\address{Department of Mathematics, Purdue University, West Lafayette, IN 47907}
\thanks{First author partly supported by a NSF FRG Grant DMS-0800428}

\author[G. Uhlmann]{Gunther Uhlmann}
\address{Department of Mathematics, University of Washington, Seattle, WA 98195}
\thanks{Second author partly supported by a NSF FRG grant No.~0554571 and a Walker Family Endowed Professorship}

\usepackage[
]{graphicx}
\graphicspath{%
    {converted_graphics/}
    {./}
    {/}
}
\usepackage{graphicx}
\begin{document}

\begin{abstract}
We give a detailed microlocal study of X-ray transforms over geodesics-like families of curves with conjugate points of fold type. We show that the normal operator is the sum of a pseudodifferential operator and a Fourier integral operator. We compute the principal symbol of both operators and the canonical relation associated to the Fourier integral operator. In two dimensions, for the geodesic transform, we show that there is always a cancellation of singularities to some order, and we give an example  where that order is infinite; therefore the normal operator is not microlocally invertible in that case. In the case of three dimensions or higher if the canonical relation is a local canonical graph we show microlocal invertibility of the normal operator. 
Several examples are also studied.
\end{abstract}

\maketitle

\section{Introduction} The purpose of this paper is to study X-ray type of transforms over geodesics-like families of curves with caustics (conjugate points). We  concentrate on the most common type of caustics --- those of fold type. Let $\gamma_0$ be a fixed geodesic segment on a Riemannian manifold, and let $f$ be a function which  support does not contain the endpoints of $\gamma_0$. The question that we are trying to answer is the following: what information about the wave front set $\WF(f)$ of $f$ can be obtained from the assumption that (possibly  weighted) integrals 
\be{01}
Xf(\gamma) =\int_\gamma f\, \d s
\ee
of $f$ along all geodesics $\gamma$ close enough to $\gamma_0$ vanish (or depend smoothly on $\gamma$)? Since $X$ has a Schwartz kernel with singularities of conormal type, $Xf$ could only provide information for $\WF(f)$ near the conormal bundle $\mathcal{N}^*\gamma_0$ of $\gamma_0$. 
If there are no conjugate points along $\gamma_0$, then we know that $\WF(f)\cap \mathcal{N}^*\gamma_0=\emptyset$. This has been shown, among the other results, in \cite{FSU,SU-AJM} in this context. It also follows from the microlocal approach to Radon transforms initiated by Guillemin \cite{guillemin} when the Bolker condition (in our case that means no conjugate points) is satisfied. Then the localized normal operator $N_\chi := X^*\chi X$, where $\chi$ is a standard cut-off near $\gamma_0$ is a pseudo-differential operator (\PDO ), elliptic at conormal directions to $\gamma_0$. If there are conjugate points along $\gamma_0$, then $N_\chi$ is no longer a \PDO. One of the goals  of this work is first to study the microlocal structure of $N_\chi$ in presence of fold conjugate points, and then use it to see what singularities can be recovered.   That would also  allow us to tell whether the problem of inverting $X$ is Fredholm or not, and would help us to determine the size of the kernel, and to analyze the stability and the possible instability of this problem. 

Geodesic X-ray transforms have a long history, generalizing the Radon type X-ray transform in the Euclidean space, see, e.g., \cite{Helgason-Radon}. When the weight is constant, and $(M,g)$ is a simple manifold with boundary, uniqueness and non-sharp stability estimates have been proven in  \cite{Mu2, MuRo, BGerver}, using the energy method. Simple manifolds are compact manifolds diffeomorphic to a ball with  convex boundary  and no conjugate points. The uniqueness result has been extended to not necessarily convex manifolds under the no-conjugate points assumption in \cite{Dairbekov}. The authors used microlocal methods to prove a sharp stability estimate in \cite{SU-Duke} for simple manifolds and  uniqueness and stability estimates for more general weighted geodesic-like transforms without conjugate points in  \cite{FSU}. The X-ray transform over magnetic geodesics with the simplicity assumption was studied in \cite{St-magnetic}. Many of those and other works study integrals of tensors as well and the results for tensors of order two or higher are less complete. 

The authors considered in \cite{SU-AJM} the X-ray transform of functions and tensors  on manifolds with possible conjugate points. Using the overdeterminacy of the problem in dimensions $n\ge3$, we showed that if there exists a family of geodesics without conjugate points with a conormal bundle covering $T^*M$, then we still have generic uniqueness and stability. In dimension two however that family has to be the set of all geodesics, and even in higher dimensions, \cite{SU-AJM} does not answer the question what is the contribution of the conjugate points to $Xf$. 

\section{Formulation of the problem} \label{sec_2}
Let $(M,g)$ be an $n$-dimensional  Riemannian manifold. Let $\exp_p(v)$, where $(p,v)\in TM$, be a regular exponential map, see section~\ref{sec_RM}, where we recall the definition given by Warner in \cite{Warner_conjugate}. The main example is the exponential map of $g$ or that of another metric on $M$ or other geodesic-like curves, for example magnetic geodesics, see also \cite{St-magnetic}.  Let $\kappa$ be a smooth function  on $TM\setminus 0$. We define the weighted X-ray transform $Xf$ by
\be{1.1}
Xf(p,\theta)= \int \kappa\big(\exp_p(t\theta), \dot \exp_p(t\theta)\big) f(\exp_p(t\theta))\, \d t, \quad (p,\theta)\in
SM,
\ee
where we used the notation
\[
\dot\exp(tv) = \frac{\d}{\d t}\exp(tv). 
\]
The $t$ integral above is carried over the maximal interval, including $t=0$,  where $\exp(t\theta)$ is defined. The assumptions that we make below guarantee that this interval remains bounded.

Let $(p_0,v_0)\in TM$ be such that $v=v_0$ is a critical point for $\exp_{p_0}(v)$ (that we call a conjugate vector) of fold type, see the definition below. Let $q_0=\exp_{p_0}(v_0)$. 
Then our goal is to study $Xf$ for $p$ close to $p_0$ and $\theta$ close to $\theta_0 := v_0/|v_0|$ under the assumption that the support of $f$ is such that $v_0$ is the only conjugate vector $v$ at $p_0$ so that $\exp_{p_0}(v)\in\supp f$. 
Note that $v_0$ can be written in two different ways as $t\theta_0$, $|\theta_0|=1$, with $\pm t>0$, and we chose the first one.  The contribution of the second one can be easily derived from our results by replacing $\theta_0$ by $-\theta_0$.

Instead of studying $X$ directly, we study the operator
\be{1.2}
\begin{split}
Nf(p) &= \int_{S_pM}   {\kappa^\sharp}(p,\theta)Xf(p,\theta)\, \d \sigma_p(\theta)\\
& = \int_{S_pM} \int  {\kappa^\sharp}(p,\theta) \kappa\left(\exp_p(t\theta), \dot \exp_p(t\theta)\right)f(\exp_p(t\theta)) \, \d t\, \d \sigma_p(\theta),
\end{split}
\ee
with some  smooth $ {\kappa^\sharp}$ localized in a  neighborhood of $(p_0,\theta_0)$. Here $\d\sigma_p(\theta)$ is the induced Riemannian surface measure on $S_p(M)$. 
When $\exp$ is the geodesic exponential map, there is a natural way to give a structure of a manifold to all non-trapping geodesics with a natural choice of a measure, see section~\ref{sec_5}.  The operator  $X$ can be viewed as map from functions or distributions on $M$ to functions or distributions on the geodesics manifold. Then one can define the adjoint $X^*$ with respect to that measure. Then the operator $X^*X$ is of the form \r{1.2} with $ {\kappa^\sharp}=\bar\kappa$, see \r{Xstar}. The condition that $\supp {\kappa^\sharp}$ should be contained in a small enough neighborhood of $(p_0,\theta_0)$ can be easily satisfied by localizing $p$ near $p_0$, and choosing $\supp\kappa$ to be near $(\gamma_{p_0,\theta_0},\dot \gamma_{p_0,\theta_0})$. 
In the case of general regular exponential maps $N$ is not necessarily $X^*X$.

A direct calculation, see \cite{SU-Duke} and Theorem~\ref{thm_Duke}, shows that the Schwartz kernel of $X^*X$ in the geodesic case (see also \cite{FSU} for general families of curves), is singular at the diagonal, as can be expected, and that singularity defines a \PDO\ of order $-1$ similarly to the integral geometry problem for geodesics without conjugate points. We refer to section~\ref{sec_5} for more details.  Next, singularities away from the diagonal exist at pairs $(p,q)$ so that $q=\exp_p(v)$ for some $v$, and $\d_v\exp_p$ is not an isomorphism ($p$ and $q$ are conjugate points). The main goal of this paper is to study the contribution of those conjugate points to the structure of $X^*X$ and the consequences of that.  We actually study a localized version of this; for a global version on a larger open set, under the assumption that all conjugate points are of fold type, one can use a partition of unity.

Let $\mathcal{U}$ be a small enough neighborhoods of $(p_0,\theta_0)$ in $SM$. Let $U$ be a small neighborhood of $p_0$ so that $U\subset \pi(\mathcal{U})$, where $\pi$ is the natural projection on the base. Fix $ {\kappa^\sharp}\in C_0^\infty(\mathcal{U})$. 
Let $Nf$ be as in \r{1.2}, related to $ {\kappa^\sharp}$, where $\kappa$ is a smooth weight.  
We will apply $X$ to functions $f$ supported in an open set $V\ni p_0$ satisfying the conjugacy assumption of the theorem below, see Figure~\ref{fig:caustics3}.  Our  goal is to study the contribution of a single fold type of singularity. Let $\Sigma\subset M\times M$ be the conjugate locus in a neighborhood of $(p_0,q_0)$, see section~\ref{sec_RM}. Finally, let $\gamma_0 = \gamma_{p_0,\theta_0}(t)$, $t\in I$, be the geodesic through $(p_0,\theta_0)$ defined in the interval $I\ni 0$, with endpoints outside $V$.

\begin{figure}[htbp] 
  \centering
  \includegraphics[bb=0 0 536 179,width=4.17in,height=1.39in,keepaspectratio]{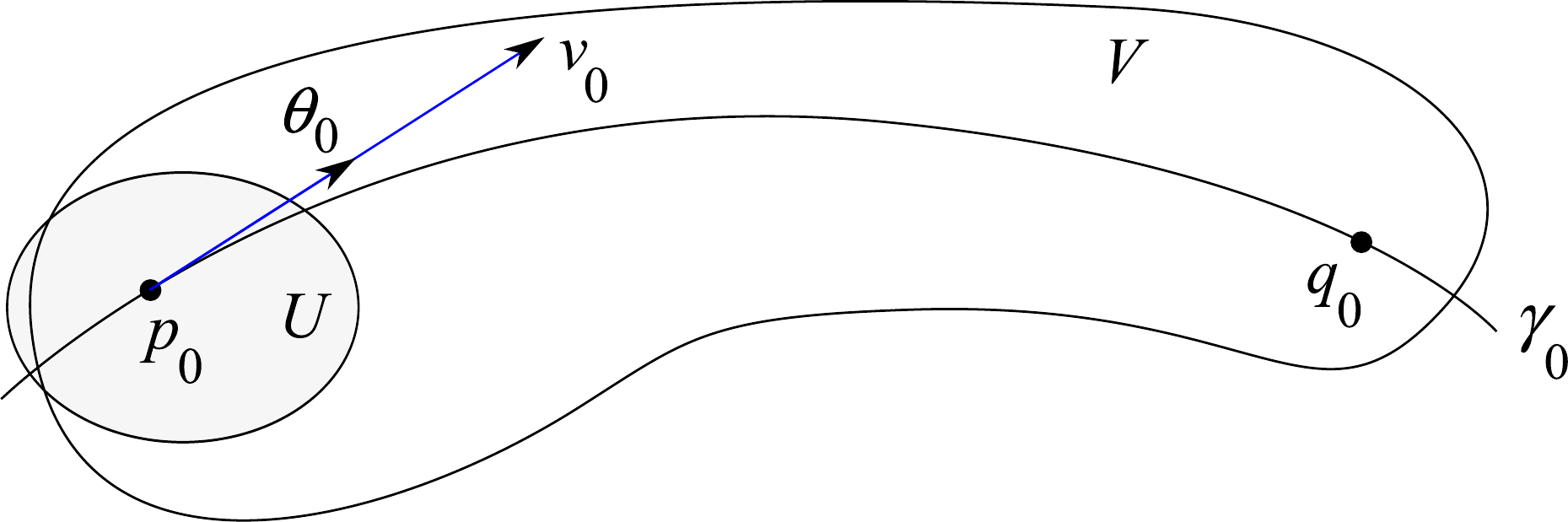}
  \caption{}
  \label{fig:caustics3}
\end{figure}

The first main result of this paper is the following.

\begin{theorem}  \label{thm_main}
Let $v_0=|v_0|\theta_0$ be a fold conjugate vector at $p_0$, and let $N$ be as in \r{1.2}.  Let 
$v_0$ be the only  singularity of $\exp_{p_0}(v)$ on the ray $\{\exp_p(t\theta_0), \; t\in I\}\cap V$. Then if \ $\mathcal{U}$ (and therefore, $U$) is  small enough, 
the operator 
\[
N : C_0^\infty(V) \longrightarrow  C_0^\infty(U)
\] 
admits the decomposition
\be{dec}
N = A +F,
\ee
where $A$ is a \PDO\ of order $-1$ with principal symbol
\be{6.1}
\sigma_p(A)(x,\xi)  = 2\pi \int_{S_xM}\delta(\xi(\theta)) \, ( {\kappa^\sharp}\kappa)(x,\theta) \,\d\sigma_x(\theta),
\ee
and $F$ is an FIO of order $-n/2$ associated to the Lagrangian $\mathcal{N}^*\Sigma$.  In particular, the canonical relation $\mathcal{C}$ of $F$ in local coordinates is given by
\be{C}
\mathcal{C} = \left\{ (p,\xi,q,\eta), (p,q)\in \Sigma, \;  \xi=-\eta_i \partial \exp_p^i(v)/\partial p, \eta\in \Coker \d_v\exp_p(v) , \; \det \d_v\exp_p(v)=0    \right\}.
\ee

If $\exp$ is the exponential map of $g$, then $\mathcal{C}$ can also be characterized as $\mathcal{N}^*\Sigma'$, where $\mathcal{N}\Sigma$ is as in \r{2.5N}, and the prime means that we replace $\eta$ by $-\eta$. 
\end{theorem}

It is easy to check that $\mathcal{C}$ above is invariantly defined. 

In section~\ref{sec_9} we show that in dimension 3 or higher in the case that  $\mathcal{C}$ is a local canonical graph the operator $N$ is microlocal invertible. In two dimensions, in the geodesic case, we show that there is always a loss of some derivatives at least when the curves are geodesics. We study in detail the case of the circular Radon transform in two dimensions in section~\ref{sec_ex}, and show that then $N$ is not microlocally invertible.

\section{Regular exponential maps and their generic singularities}  \label{sec_RM}
\subsection{Regular exponential maps} 
Let $M$ be a fixed $n$-dimensional manifold. We will recall the definition of Warner \cite{Warner_conjugate} of a regular exponential map at $p\in M$. We think of it as a generalization of the exponential map on a Riemannian manifold, by requiring only those properties that are really necessary for what follows. For that reason, we use the notation $\exp_p(v)$. In addition to  \cite{Warner_conjugate} , we will require $\exp_p(v)$ to be smooth in $p$ as well. Let $N_p(v)\subset T_vT_pM$ denote the kernel of $\d \exp_p$. Unless specifically indicated, $\d$ is the differential w.r.t.\ $v$. The radial tangent space at $v$ will be denoted by $r_v$. It can be identified with  $\{sv, \;  s\in\R\}$, where $v$ is considered as an element of $T_vT_pM$. 

\begin{definition}
A map $\exp_p(v)$ that for each $p\in M$ maps $v\ni T_pM$ into M is called a \textbf{regular exponential map}, if 
\begin{itemize}
  \item[\textbf{(R1)}] $\exp$ is smooth in both variables, except possibly at $v=0$.  Next,     $\d\exp_p(tv)/\d t \not=0$, when $v\not=0$. 
  \item[\textbf{(R2)}] The Hessian $\d^2\exp_p(v)$ maps isomorphically $r_v\times N_p(v)$ onto $T_{\exp_p(v)}M/\d \exp_p (T_vT_pM)$ for any $v\not=0$ in $T_pM$ for which $\exp_p(v)$ is defined. 
  \item[\textbf{(R3)}] For each $v\in T_pM\setminus 0$, there is a convex neighborhood $U$ of $v$ such that the number of singularities of $\exp_p$, counted with multiplicities, on the ray $tv$, $t\in \R$ in $U$, for each such ray that intersects $U$, is constant and equal to the order of $v$ as a singularity of $\exp_p$.  
\end{itemize}    
\end{definition}

An example is the exponential map on a Riemannian (or more generally on a Finsler manifold), see \cite{Warner_conjugate}. Then (R1) is clearly true. Next, (R2) follows from the following well known property. Fix $p$ and a geodesic through it. Consider all  Jacobi fields vanishing at $p$. Then at any $q$ on that geodesic, the values of those Jacobi fields that do not vanish at $q$ and the covariant derivatives of those that vanish at $q$ span $T_qM$. Also, those two spaces are orthogonal. Finally, (R3) represents the well known continuity property of the conjugate points, counted with their multiplicities that follows from the Morse Index Theorem (see, e.g.,  \cite[Thm~4.3.2]{Jost}). 

We would need also an assumption about the behavior of the exponential map at $v=0$. 
\begin{itemize}
  \item[\textbf{(R4)}]   $\exp_p(tv)$  is  smooth in $p,t,v$ for all $p\in M$, $|t|\ll1$, and $v\not=0$. Moreover,
\[  
\exp_p(0)=p, \quad \mbox{and} \quad 
\frac{\d}{\d t}\exp_p(tv)=v \quad \mbox{for $t=0$}.
\] 
\end{itemize}    
Given a regular exponential map, we define the ``geodesic'' $\gamma_{p,v}(t)$, $v\not=0$, by $\gamma_{p,v}(t)=\exp_p(tv)$. We will often use the notation 
\be{qw}
q= \exp_p(v) = \gamma_{p,v}(1), \quad w= -\dot\exp_p(v) := -\dot \gamma_{p,v}(1), \quad \theta=v/|v|. 
\ee
Note that the ``geodesic flow'' does not necessarily obey the group property. We will assume that
\begin{itemize}
  \item[\textbf{(R5)}] For $q$, $w$ as in \r{qw}, we have $\exp_q(w)=p$, $\dot\exp_q(w)=-v$.
\end{itemize}    
This shows that in particular, $(p,v) \mapsto (q,w)$ is a diffeomorphism. If $\exp$ is the exponential map of a Riemannian metric, 
then (R5) is automatically true and  that map is actually a symplectomorphism  (on $T^*M$). 

\begin{remark}
In case of magnetic geodesics, or more general Hamiltonian flows, (R5) is equivalent to time reversibility of the ``geodesics.'' This is not true in general. On the other hand, one can define the reverse exponential map $\exp^-_q(w) = \gamma_{q,-w}(-1)$ in that case, see e.g.\ \cite{St-magnetic}, near $(q_0,w_0)$, and replace $\exp$ by $\exp^-$ in that neighborhood. Then (R5) would hold. 
In other words, (R5) really says that $(p,v) \mapsto (q,w)$ is assumed to be a local diffeomorphism with an inverse satisfying (R1) -- (R4).  
\end{remark}

\subsection{Generic properties of the conjugate locus}
We recall here the main result by Warner  \cite{Warner_conjugate} about the regular points of the conjugate locus of a fixed point $p$. The \textbf{\emph{tangent conjugate locus}} $S(p)$ of $p$  is the set of all vectors $v\in T_pM$ so that $\d\exp_p( v)$ (the differential of $\exp_p(v)$ w.r.t.\ $v$) is not an isomorphism.
We call such vectors \textbf{\emph{conjugate vectors}} at $p$ (called conjugate points in \cite{Warner_conjugate}).  The kernel of $\d\exp_p(v)$ is denoted by $N_p(v)$. It is part of $T_vT_pM$ that we identify with $T_pM$. In the Riemannian case, by the Gauss lemma, $N_p(v)$ is orthogonal to $v$. In the general case, by (R1), it is always transversal to $v$. 
The images of the conjugate vectors under the exponential map $\exp_p$ will be called \textbf{conjugate points} to $p$. The image of $S(p)$ under the exponential map $\exp_p$ will be denoted by $\Sigma(p)$ and called the \textbf{conjugate locus of $p$}. Note that $S(p)\subset T_pM$, while $\Sigma(p)\subset M$. We always work with $p$ near a fixed $p_0$ and with $v$ near a fixed $v_0$. Set $q_0=\exp_{p_0}(v_0)$. Then we are interested in $S(p)$ restricted to a small neighborhood of $v_0$, and in $\Sigma(p)$ near $q_0$. Note that $\Sigma(p)$ may not contain all points near $q_0$ conjugate to $p$ along some ``geodesic''; and may not contain even all of those along $\exp_{p_0}(tv_0)$ if the later self-intersects --- it contains only those that are of the form $\exp_p(v)$ with $v$ close enough to $v_0$. 

Normally, $\d \exp_p(v)$ stands for the differential of $\exp_p(v)$ w.r.t.\ $v$. 
When we need to take the differential w.r.t.\ $p$, we will use the notation $d_p$ for it, We write $d_v$ for the differential w.r.t.\ $v$, when we want to distinguish between the two.

We denote by $\Sigma$ the set of all conjugate pairs $(p,q)$ localized as above. In other words, $\Sigma=\{(p,q);\; q\in \Sigma(p)\}$, where $p$ runs over a small neighborhood of $p_0$. Also, we denote by $S$ the set $(p,v)$, where $v\in S(p)$. 
 
A \textbf{\emph{regular conjugate vector}}  $v$ is defined by the requirement that  there exists a neighborhood  of $v$, so that any radial ray of $T_pM$ contains at most one conjugate point there.  The regular conjugate locus then is an everywhere dense open subset of the conjugate locus that has a natural structure of an $(n-1)$-dimensional manifold. The order of a conjugate vector as a singularity of $\exp_p$ (the dimension of the kernel of the differential) is called an order of the conjugate vector. 

In \cite[Thm~3.1]{Warner_conjugate}, Warner characterized the  conjugate vectors at a fixed $p_0$ of order at least $2$, and some of those of order $1$, as described below. Note that in $\B_1$, one needs to postulate that $N_{p_0}(v)$ remains tangent to $S(p_0)$ at points $v$ close to $v_0$ as the latter is not guaranteed by just assuming that it holds at $v_0$ only. 

\smallskip  
\begin{description}
\item [($\mathbf{F}$)  Fold conjugate vectors] Let $v_0$ be a regular  conjugate vector at $p_0$, and let $N_{p_0}(v_0)$ be one-dimensional and transversal to $S(p_0)$.
Such singularities are known as fold singularities.   
Then one can find local coordinates $\xi$ near $v_0$ and $y$ near $q_0$ so that in those coordinates, $\exp_{p_0}$ is given by 
\be{2.1}
y' = \xi', \quad y^n=(\xi^n)^2.
\ee
Then 
\be{2.1n}
S(p_0)=\{\xi^n=0\}, \quad N_{p_0}(v_0)=\mbox{span}\left\{ \partial/\partial \xi^n\right\}, \quad\Sigma(p_0) = \{y^n=0\}. 
\ee
Since the fold condition is stable under small $C^\infty$ perturbations, as follows directly from the definition, those properties are preserved under a small perturbation of $p_0$. 
\smallskip  

\item[($\mathbf{B}_1$) Blowdown of order 1]   
Let $v_0$ be a regular  conjugate vector at $p_0$ and let $N_{p_0}(v)$ be one-dimen\-sional. Assume also that  $N_{p_0}(v)$ is tangent to $S(p_0)$ for all regular conjugate $v$ near $v_0$. 
We call such singularities  blowdown of order 1. 
Then locally, $\exp_{p_0}$ is represented in suitable coordinates by
\be{2.1b}
y'=\xi', \quad y^n = \xi^1\xi^n.
\ee  
Then 
\be{2.1bb}
S(p_0) = \{\xi^1=0\}, \quad N_{p_0}(v_0)=\mbox{span}\left\{ \partial/\partial \xi^n\right\},  \quad \Sigma(p_0) = \{y^1=y^n=0  \}. 
\ee
Even though we postulated that the tangency condition is stable under perturbations of $v_0$, it is not stable under a small perturbation of $p_0$, and  the type of the singularity may change then. In some symmetric cases, one can check directly that the type is locally preserved. 
\smallskip  

\item[($\mathbf{B}_k$) Blowdown of higher order] Those are regular conjugate vectors in the case where $N_{p_0}(v_0)$ is  $k$-dimensio\-nal, with $2\le k\le n-1$. Then in some coordinates,  $\exp_{p_0}$ is represented as
\be{2.1c}
\begin{split}
y^i &=\xi^i,\qquad i =1,\dots,n-k\\ 
y^i &= \xi^1\xi^i, \quad i =n-k+1,\dots,n.
\end{split}
\ee
Then 
\be{2.1cn}
\begin{split}
S(p_0) &= \{\xi^1=0\}, \quad N_{p_0}(v_0)=\mbox{span}\left\{\partial/\partial{\xi^{n-k+1}},\dots,\partial/\partial{\xi^n}\right\},\\
\Sigma(p_0) &= \{y^1=y^{n-k+1}= \dots =y^n=0  \}.
\end{split}
\ee
In particular, $N_{p_0}(v_0)$ is tangent to $S(p_0)$. This singularity is unstable under perturbations of $p_0$, as well. A typical example are  the antipodal points on $S^{n}$, $n\ge3$; then $k=n-1$. 
\end{description}

The purpose of this paper is to study the effect of fold conjugate points to $X$.

\section{Geometry of the fold conjugate locus} 
In this section, we study the geometry of the tangent conjugate locus $S(p)$, and $S$ respectively; and the conjugate locus $\Sigma(p)$ and $\Sigma$, respectively. Recall that we work locally, and everywhere below, even if not stated explicitly, $(p,v)$ belongs to a small enough neighborhood of $(p_0,v_0)$; $(q,v)$ is near  $(q_0,w_0)$. We assume throughout the section that $v_0$ is conjugate vector at $p_0$ of fold type. We also fix a non-zero covector $\eta_0$ at $q_0$  as in \r{C}, and let $\xi_0$ be the corresponding $\xi$ as in \r{C}. We will see later that $\xi_0\not=0$. We refer to Figure~\ref{fig:caustics_multidim}, where $w$ is not shown, and the zero subscripts are omitted. 

We start with properties of $S(p)$ and $S$.

\begin{lemma}\label{lemma_2.1}   \ \ 

(a) Let $v\in S(p)$ be  a fold conjugate vector. Then near $q = \exp_p(v)$, $\Sigma(p)$ is a smooth surface of codimension one, tangent to $w := -\dot\gamma_{p,v}(1)$. 

(b)  $S$ is a smooth $(2n-1)$-dimensional surface in $TM$ that can be considered as the bundle  $\{S(p),p\in M\}$ with fibers $S(p)$. 
\end{lemma}

\begin{figure}[h] 
  \centering
  \includegraphics[bb=0 0 751 371,width=5.54in,height=2.74in,keepaspectratio]{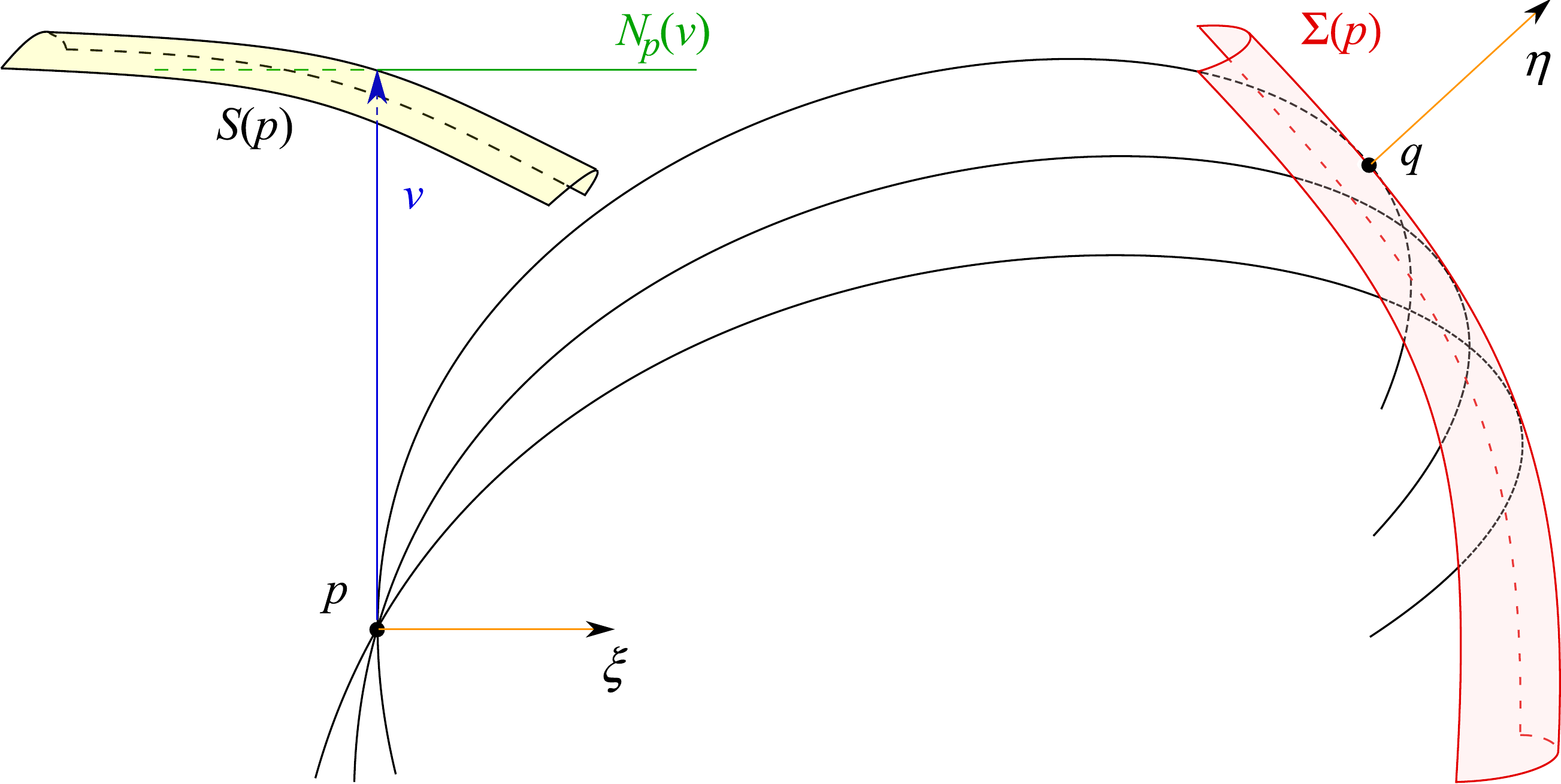}
  \caption{A typical fold conjugate locus}
  \label{fig:caustics_multidim}
\end{figure}

\begin{proof}
Consider (a) first. 
The representation \r{2.1} implies that locally,  $\Sigma(p) =\exp_p(S(p))$ is a smooth surface of codimension one (given by $y^n=0$). Next, for $v\in S(p)$, the differential $\d\exp_p$ sends any vector to a vector tangent to $S(p)$, as it follows from \r{2.1} again. In particular, this is true for the radial vector $v$ (considered as a  vector in   $T_vT_pM$). This proves that $w$ is tangent to $\Sigma(p)$.

The statement (b) follows from the fact that $S$ is defined by $\det \d \exp_p(v)=0$, and  that $\det \d \exp_p(v)$ has a non-vanishing differential w.r.t.\ $v$.  
\end{proof}

\begin{remark} It is easy to show that in (a),  $\gamma_{p,v}$ is tangent to $\Sigma(p)$ of order $1$ only.
\end{remark}

We  define ``Jacobi fields'' along $\gamma_{p,v}$ vanishing at $p$ as follows. For any $\alpha\in T_vT_pM$, set 
\[
J(t) = \d[\exp_p(tv)](\alpha) = \alpha^k\frac{\partial}{\partial v^k }\exp_p(tv).
\]
Then $J(0)=0$, $\dot J(0)=\alpha$, where $\dot J(T)= \d J(t)/\d t$. If $J(1)=0$, then a direct computation shows that 
\be{JJ}
\dot J(1) = 
\d^2\exp_p(v) (\alpha\times v). 
\ee

When $\exp$ is the exponential map of a Riemannian metric, it is natural to work with the covariant derivative $D_tJ(t)= :J'(1)$ instead of $\dot J(t)$. While they are different in general, they coincide at points where $J(t)=0$. 

The next lemma shows that the fold/blowdown conditions are symmetric w.r.t.\ $p$ and $q$. 

\begin{lemma}  \label{lemma_2.2} 
The vector $v_0$ is a conjugate vector at $p_0$ of fold type, if and only if   $w_0$ is a conjugate vector at $q_0$ of fold type. 
\end{lemma}

\begin{proof}
Set $w_0=-\dot\gamma_{p_0,v_0}(1)$, as in \r{qw}. Then $p_0=\exp_{q_0}(w_0)$. Assume now that $\alpha\in N_{p_0}(v_0)$. In some local coordinates,  differentiate $p=\exp_q(w)$ w.r.t.\ $v$ in the direction of $\alpha$; here $q$, $w$ are viewed as functions of $p$, $v$. Then, using the Jacobi field notation introduced above in \r{JJ},  we get
\[
0=\d\exp_{q_0}(w_0)\left(\alpha^k \frac{\partial w}{\partial v^k}(p_0,v_0)\right)= \d\exp_{q_0}(w_0)\dot J(1)
\]
because
\[
\alpha^k \frac{\partial w}{\partial v^k}(p_0,v_0) =
\alpha^k \frac{\partial }{\partial v^k} \frac{\d}{\d t}\Big|_{t=1} \exp_p(tv)       (p_0,v_0) 
= \dot J(1).
\]
By (R2), $\dot J(1)\not=0$, so in particular, this shows that $w_0$ is conjugate at $q_0$, and $\dot J(1)\in N_{q_0}(w_0)$. 
Moreover, by (R2), the linear map 
\be{2.4}
N_p(v)\ni \alpha = \dot J(0) \mapsto \dot J(1) :=\beta\in N_q(w), \quad J(0)=J(1)=0
\ee
defines an isomorphism between $N_p(v)$ and $N_q(w)$. 
 Then \r{2.4} shows that $w_0$ is conjugate at $q_0$ of  multiplicity one. By (R3), applied to $w_0$, it is also regular. 

We will prove now that $w_0$ is of fold type. Since it is regular and of multiplicity one, $S(q_0)$ near $w_0$ is a smooth $(n-1)$ dimensional surface either of  type $\F$, as in \r{2.1n} or of   type $\B_1$, as in \r{2.1bb}. Assume the latter case first, then $\Sigma(q_0)$ is of codimension two, as follows from \r{2.1bb}. In particular, using the normal form \r{2.1b}, we see that in this case, one can find a non-trivial one-parameter family of vectors $w(s)$ so that $w(0)=w_0$ and $\exp_{q_0}(w(s))=p_0$. Then the corresponding tangent vectors at $p_0$ would form a non-trivial one-parameter family of vectors $v(s)$ so that $\exp_{p_0}(v(s))=q_0$. That cannot happen, if $v_0$ is of type $\F$, see \r{2.1}, since the equation $\exp_{p_0}(v)=q_0$ has (near $v_0$) at most two solutions. 
\end{proof}

For $(p,v)\in S$, let $\alpha=\alpha(p,v)\in N_p(v)$ be a unit vector.  To fix the direction, assume that the derivative of $\det \,\d \exp_p(v)$  in the direction of $\alpha$, for $v$ a conjugate vector,  is positive. 
 Here we identify in $T_vT_pM $ and $T_pM$. In the fold case, $N_p(v)$ is clearly a smooth vector bundle on $TM$ near $(p_0,v_0)$, and  $\alpha$ is a smooth vector field. 

\begin{lemma} \label{lemma_graph}
For any fixed $p$ near $p_0$, the map 
\be{LG}
S(p)\ni v\mapsto \alpha(p,v)\in N_p(v)
\ee
is a local diffeomorphism, smoothly depending on $p$ if and only if
\be{cond}
\d^2\exp_{p_0}(v_0) \left( N_{p_0}(v_0)\setminus 0\times\ \cdot\  \right)\big|_{T_{v_0}S(p_0)}   \quad \mbox{is of full rank}.
\ee
\end{lemma}

\begin{proof} In local coordinates, we want to find a condition so that the equation 
\[
\alpha^i\partial_{v^i} \exp_{p}(v) =0
\]
can be solved for $v$ so that $v=v_0$ for $(p,\alpha)=(p_0,\alpha_0)$, where  $\alpha_0=\alpha(p_0,v_0)$. Then $v$ would automatically be in $S(p)$. By the implicit function theorem, this is equivalent  to  
\[
\det\left(  \partial_v  \alpha^i_0\partial_{v^i} \exp_{p_0}(v) \right) \not=0 \quad  \mbox{at $v=v_0$}.
\] 
Choose a coordinate system near $v_0$ so that $\partial/\partial v^n$  spans $N_{p_0}(v_0)$, and $\{\partial/\partial v^1, \dots \partial/\partial v^{n-1}\}$ span $T_{v_0}S(p_0)$. Denote $F(v)= \exp_{p_0}(v)$ and denote by $F_i$, $F_{ij}$ the corresponding partial derivatives. Greek indices below run from $1$ to $n-1$. We have
\begin{align} \label{2.8.1}
\partial_n F(v_0) &= 0, &&\qquad \mbox{because $   \partial/\partial v^n   \in N_{p_0}(v_0)$},\\ \label{2.8.2}
 \partial_\alpha \det (\partial F)(v_0)&=0,& &\qquad \mbox{because $   \partial/\partial v^\alpha$ is tangent to $S(p_0)$ at $v_0$},\\   \label{2.8.3}
 \partial_n    \det (\partial F)(v_0)  &\not = 0,  &&\qquad \mbox{by the fold condition},\\ \label{2.8.4}
c^\alpha \partial_\alpha F(v_0)  &\not=0, \quad \forall c\not=0, 
&&\qquad \mbox{because $   c^\alpha\partial/\partial v^\alpha\not\in N_{p_0}(v_0)$.}
\end{align}
We want to prove that $\det (\partial_n \partial F)(v_0)\not=0$ if and only if \r{cond} holds. That determinant equals 
\be{det}
\det(F_{1n}, F_{2n},\dots, F_{nn})(v_0).
\ee
Perform the differentiation in \r{2.8.2}. By \r{2.8.1}, \r{2.8.4},
\[
\det(F_{1}, \dots,F_{n-1}, F_{n\alpha})(v_0)=0, \quad \forall \alpha  \quad \Longrightarrow \quad F_{n\alpha}(v_0)\in \mbox{span}(F_1(v_0),\dots, F_{n-1}(v_0)).
\]
Similarly, \r{2.8.3} implies
\be{2.8.6}
\det(F_{1}, \dots,F_{n-1},  F_{nn})(v_0)\not=0\quad \Longrightarrow \quad 0\not=F_{nn}(v_0)\not\in \mbox{span}(F_1(v_0),\dots, F_{n-1}(v_0)).
\ee
Those two relations show that \eqref{det} vanishes if and only if $(F_{n1}(v_0),\dots F_{n, n-1}(v_0))$ form a linearly dependent system, that is equivalent to \eqref{cond}.
\end{proof}

We study the structure of the conjugate locus $\Sigma(p)$, $\Sigma(q)$ and $\Sigma$ next. Recall again that we work locally near $p_0$, $v_0$ and $q_0$. 

\begin{theorem}  \label{thm_F}  Let $v_0$ be a fold conjugate   vector at $p_0$.  

(a) Then for any $p$ near $p_0$, $\Sigma(p)$ is a smooth hypersurface of dimension $n-1$ smoothly depending on~$p$. Moreover for any $q=\exp_p(v) \in\Sigma(p)$, $T_qM$ is a direct sum of the linearly independent  spaces
\be{tq}
T_qM=   T_q\Sigma(p)\oplus N_q(w),
\ee
and
\[
T_q\Sigma(p) = \mbox{\rm Im}\; \d \exp_p(v), \quad N_q^*\Sigma(p) = \Coker \d_v\exp_p(v).
\]
Next, those statements remain true with  $p$ and $q$  swapped.

(b) 
$\Sigma$ is a smooth $(2n-1)$-dimensional hypersurface in $M\times M$ near $(p_0,q_0)$, that is also a fiber bundle  $\Sigma =\{\Sigma(p),\; p\in M\}$ with fibers $\Sigma(p)$ (and also $\Sigma=  \{\Sigma(q),\; q\in M\}$).  
Moreover, the conormal bundle $\mathcal{N}^*\Sigma$ is given by
\be{2.5cn}
\begin{split}
\mathcal{N}^*\Sigma = &\big\{ (p,q,\xi,\eta);\; (p,q)\in \Sigma, \,  \xi=\eta_i \partial \exp_p^i(v)/\partial p, \eta\in \Coker \d_v\exp_p(v) \\
&\quad \mbox{where $v=\exp_p^{-1}(q)$ with $\exp_p$ restricted to $S(p)$        }  \big\}.
\end{split}
\ee
\end{theorem}

\begin{proof}
We start with (a). 
By the normal form \r{2.1}, also clear from the fold condition, 
the image of $S(p)$ under $\d\exp_p(v)$ coincides with $T_q\Sigma(p)$. 
In particular, $\d\exp_p(v)$, restricted to $S(p)$ is a diffeomorphism to its image. Relation \r{tq} follows from \r{2.4} and (R2).

Consider  (b). We have $(p,q)\in \Sigma$ if and only if there exists $v$ (near $v_0$) so that 
\be{2.6}
q=\exp_p(v), \quad \det\d_v \exp_p(v)=0.
\ee
In some local coordinates, we view this as $n+1$ equations for the $3n$-dimensional variable $(p,q,v)$ near $(p_0,q_0,v_0)$. We show first that the solution that we denote by $L$, is a $(2n-1)$-dimensional submanifold. To this end, we need to show that the following differential has rank $n+1$ at $(p_0,q_0,v_0)$:
\be{2.7}
\begin{pmatrix}
\d_p \exp_p(v)&-\Id & \d_v \exp_p(v)\\
\d_p \det \d_v\exp_p(v)& 0& \d_v \det \d_v \exp_p(v)
\end{pmatrix}.
\ee
The elements of the first ``row'' are $n\times n$ matrices, while the second row consists of three $n$-vectors. That the rank of the differential above is full follows from the fact that $\d_v \det \d_v \exp_p(v)\not=0$ at $(p_0,v_0)$, guaranteed by the fold condition. 

Set $\pi(p,q,v)=(p,q)$. We  show next that $\pi(L)$ is a $(2n-1)$-dimensional submanifold, too. To this end, we need to show that $\d\pi$ is injective on $TL$. The tangent space to $L$ is given by the orthogonal complement to the rows of \r{2.7}. Let us denote any vector in $TL$ by $\rho =(\rho_p,\rho_q,\rho_v)$. Then $\d\pi(\rho) = (\rho_p,\rho_q)$. Our goal is therefore to show that $\rho_p=\rho_q=0$ implies $\rho_v=0$. Then $(0,0,\rho_v)$ is orthogonal to the rows of \r{2.7}, therefore,
\[
\rho_v^i\partial_{v^i}\exp_p^k (v)=0, \quad k=1,\dots,n, \quad 
\rho_v^i\partial_{v^i}\det \d_v \exp_p(v)=0. 
\]
The latter identity shows that $\rho_v\in N_p(v)$, while the first one shows that $\rho_v\in \Ker \d_v \exp_p(v)$. By the fold condition, $\rho_v=0$. 

This analysis also shows that the covectors $\nu$ orthogonal to $\Sigma$ are of the form $\nu=(\nu_p,\nu_q)$ with the property that $(\nu_p,\nu_q,0)$ is conormal to $L$. Since the conormals  to $L$ are spanned by the rows of \r{2.7}, in order to get the third component to vanish, we have to take a linear combination with coefficients $a_i$, $i=1,\dots,n$ and $b$ so that 
\be{2.5a}
a_i\frac{\partial q^i}{\partial v^j}+b \frac{\partial \det \d_v \exp_p(v) }{\partial v^j}=0, \quad \forall j,
\ee
where $q=\exp_p(v)$. Let $0\not=\alpha\in N_p(v)$. Multiply by $\alpha^j$ and sum over $j$ above to get that the $v$-derivative of $b\det \d_v \exp_p(v)$ in the direction of $N_p(v)$ vanishes. According to the fold assumption, this is only possible if $b=0$. 
Then we get that $a\in \Coker\d_v \exp_p(v)$. 
Therefore the normal covectors to $\Sigma$ are of the form
\be{2.8}
\nu = \left(\left\{a_i \frac{\partial q^i}{\partial p^j}\right\}, -a\right), \quad a\in   \Coker\d_v \exp_p(v), 
\ee
that proves \r{2.5cn}. 

\end{proof}

\begin{theorem} Let $v_0$ be a fold conjugate vector at $p_0$. 
Let $\exp_p$ be the exponential map of a Riemannian metric. 

(a) Then the sum in \r{tq}  is an orthogonal one, i.e.,
\[
N_q\Sigma(p) = N_q(w). 
\]

(b) Next, \r{2.5N} also admits the representation
\be{2.5N}
\begin{split}
\mathcal{N}\Sigma = &\big\{ (p,q,\alpha,\beta);\; (p,q)\in \Sigma, \,  \alpha =J'(0), \beta= -J'(1),\mbox{ where $J$ is any Jacobi field}\\
&\quad \mbox{along the locally unique geodesic connecting $p$ and $q$ with  $J(0)=J(1)=0$}  \big\}.
\end{split}
\ee

(c) $\mathcal{N}\Sigma$ is a graph of a smooth map $(p,\alpha)\mapsto (q,\beta)$ if and only if condition \r{cond} is fulfilled. Then that maps is a local diffeomorphism. 
\end{theorem}

\begin{remark}
Note that for $(p,q)\in \Sigma$, the geodesic connecting $p$ and $q$ is unique,  as follows from the normal form \r{2.1}, only among the geodesics with $\dot\gamma(0)$ close to $v_0$. Also, $J$ is determined uniquely up to a multiplicative constant. 
Next, once we prove that $\Sigma$ is smooth, then $\alpha\in N_p(v)$ and $\beta\in N_q(w)$ by (a) (see also \r{2.1}), but \r{2.5N} gives something more than that --- it restricts $(\alpha,\beta)$ to an one-dimensional space. 
\end{remark}

\begin{remark}
It is a natural question whether $|J'(0)|=|J'(1)|$. One can show that generically, this is not the case. 
\end{remark}

\begin{proof}
By \cite[Lemma~IX.3.5]{Lang_manifolds},  the conjugate of $\d\exp_p(v)$, w.r.t.\ the metric form is given by
\be{2.5}
\left(\d\exp_p(v)\right)^* =  \d\exp_q(w),
\ee
where we use the notation \r{qw}. 
The normal to $\Sigma(p)$ at $q$  is in the orthogonal complement to the image of $\d\exp_p(v)$, that by \r{2.5} is   $\Ker \d\exp_q(w) =N_q(w)$. This proves (a).

  Then we get by \r{2.5}, \r{2.5a} (where $b=0$) that $a\in N_q(w)$, where we identify the covector $a$ with a vector by the metric. 

We will use now \cite[Lemma~IX.3.4]{Lang_manifolds}: for any two Jacobi fields $J_1$, $J_2$ along a fixed geodesic, the Wronskian  $\langle J_1',J_2\rangle -\langle J_1,J_2'\rangle $ is constant.  Along the geodesic connecting $p$ and $q$, in fixed coordinates near $p$, let $\tilde J$ be determined by $\tilde J(0)=e_j$, $\tilde J'(0)=0$. Here $e_j$ has components $\delta_j^i$. 
If $p$ and $q$ are conjugate to each other, then $\tilde J(1)$ is the 
equal to the variation $\partial q/\partial p^j$, and this is independent on the choice of the local coordinates, as long as $e_j$ is considered as a fixed vector at $p$. 
Define another Jacobi field by $J(1)=0$, $J'(1)=a$, where $a$ is as in \r{2.8} but considered as a vector. Denote  the field in the brackets in \r{2.8} by $X_j$. Then
\begin{align}  \label{2.9}
X_j &= \langle a,\tilde J(1)\rangle\\ \nonumber &= \langle J'(1),\tilde J(1)\rangle \\ \nonumber &=
\langle J'(1),\tilde J(1)\rangle- \langle J(1),\tilde J'(1)\rangle\\ \nonumber
&=\langle J'(0),\tilde J(0)\rangle- \langle J(0),\tilde J'(0)\rangle\\ \nonumber
&= J_j'(0).
\end{align}
This proves \r{2.5N}.

The proof of (c) follows directly from Lemma~\ref{lemma_graph}.
\end{proof}

\section{The Schwartz kernel of  $N$ near the diagonal and  mapping properties of $X$ and $N$}\label{sec_5}
\subsection{The geodesic case}
Let $\exp$ be the exponential map of the metric $g$. Then $X$ is the weighted  geodesic ray transform. One way to parametrize the geodesics is the following. Let $H$ be any orientable hypersurface with the property that it intersects transversally, at one point only, any geodesic in $\Omega$ issued from a point in $\mathcal{U}$. For our local analysis, $H$ can be an arbitrarily small surface intersecting transversally $\gamma_{p_0,v_0}$, so let us fix that choice. 
Let $\d\Vol_H$ be the induced measure in $H$, and let $\nu$ be a smooth unit normal vector field on $H$ consistent with the orientation of $H$. Let $\mathcal{H}$ consists of all $(p,\theta)\in SM$ with the property that $p\in H$ and $\theta$ is not tangent to $H$, and positively oriented, i.e., $\langle\nu,\theta \rangle>0$.  Introduce the measure $\d\mu = \langle n,\theta\rangle\,\d\Vol_H(p)\,\d \sigma_p(\theta) $ on $\mathcal{H}$. 
Then one can parametrize all geodesics intersecting $H$ transversally by their intersection $p$ with $H$ and the corresponding direction, i.e., by elements in $\mathcal{H}$. An important property of $\d\mu$ is that it introduces a measure on that geodesics set that is invariant under a different choice of $H$ by the Liouville Theorem, see e.g., \cite{SU-Duke}. 

The weighted geodesic transform $X$ can be defined as in \r{1.1} for $(p,\theta)\in \mathcal{H}$ instead of $(p,\theta)\in \mathcal{U}$ because transporting $(p,v)$ along the geodesic flow does not change the integral. Since we assumed originally that $\kappa$ is localized near a small enough neighborhood of $\gamma_{p_0,v_0}$, we get that $\kappa$ is supported in a small neighborhood of $(p_0,\theta_0)$ in $\mathcal{H}$. We view $X$ as the following map
\[
X :L^2(M) \to L^2(\mathcal{H},\d\mu),
\]
restricted to a neighborhood of $(p_0,\theta_0)$. This map is bounded, see \cite{Sh-book}, and this also follows from our analysis of $N$. By  the proof of Proposition~1 in \cite{SU-Duke}, $X^*X$  is given by 
\be{Xstar}
X^*Xf(p) =  \frac1{\sqrt{\det g(p)}} \int_{S_pM}\int\bar\kappa(p,\theta) \kappa\big(\exp_p(t\theta), \dot \exp_p(t\theta)\big) f(\exp_p(t\theta))\, \d t\, \d\sigma_p(\theta).
\ee
We therefore proved the following. 
\begin{proposition}\label{pr_N}
Let $\exp$ be the geodesic exponential map. 
Let $X$ be the weighted geodesic ray transform \r{1.1}, and  let $N$ be as in \r{1.2}, depending on $ {\kappa^\sharp}$. Then
\[
X^*X = N\quad \mbox{with $ {\kappa^\sharp}=\bar\kappa$}.
\]
\end{proposition}

Split the $t$ integral in \r{Xstar} in two: for $t>0$ and for $t<0$, and make a change of variables $(t,\theta)\mapsto(-t,-\theta)$ in the second one to get
\be{Xstar2}
X^*Xf(p) =  \frac1{\sqrt{\det g(p)}} \int_{T_pM} W(p,v) f(\exp_p(v))\,\d\Vol(v),
\ee
where
\be{W}
\begin{split}
W  &= |v|^{-n+1}\Big(\bar\kappa(p,v/|v|)\kappa \big(\exp_p(v), \dot \exp_p(v)/| v| \big) \\
& \qquad\quad \qquad + \bar\kappa(p,-v/|v|)\kappa \big(\exp_p(v), -\dot \exp_p(v)/| v| \big)\Big). 
\end{split}
\ee
Note that $|\dot\exp_p(v)|=|v|$ in this case. 

Next we recall a result in \cite{SU-Duke}. Part (a) is based on formula \r{Xstar2} after a change of variables. 

\begin{theorem}[\cite{SU-Duke}] \label{thm_Duke}
Let $\exp$ be the exponential map of $M$. Assume that $\exp_p : \exp_p^{-1}(M)\to M$ is a diffeomorphism for $p$ near $p_0$. 

(a) Then for $p$ in the same neighborhood of $p_0$,
\be{7.1}
X^*Xf(p) = \frac1{\sqrt{\det g(p)}}\int  A(p,q)\frac{f(y)}{\rho(p,q)^{n-1}}\Big| \det \frac{\partial^2(\rho^2/2)}{\partial p \partial q}\Big|\,\d q,
\ee
where
\[
A(p,q) = \bar \kappa(p, -\grad_p\rho)\kappa(q,\grad_q\rho) + 
\bar \kappa(p, \grad_p\rho)\kappa(q,-\grad_q\rho).
\]

(b) $X^*X$ is a classical \PDO\ of order $-1$ with principal symbol
\be{a2}
\sigma_p(X^*X)(x,\xi)  = 2\pi \int_{S_xM}\delta(\xi(\theta))  |\kappa(x,\theta)|^2  \,\d\sigma_x(\theta),
\ee
where $\xi(\theta)=\xi_i\theta^j$, and $\delta$ is the Dirac delta function.
\end{theorem}

Note that the integral \r{7.1} is not written in an invariant form but one can easily check that writing it w.r.t.\ the volume form, the kernel is invariant. We also note that in the proof of Theorem~\ref{thm_main}, we apply the theorem above by restricting $\supp f$ and the region where we study $Nf$ to a small enough neighborhood of $p_0$, where we there will be no conjugate points. This gives the \PDO\ part $A$ of $N$ in Theorem~\ref{thm_main}.

\bigskip
\paragraph{\bf Mapping properties of $X$}\label{sec_map} 
Let $(x',x^n)$ be semigeodesic coordinates on $H$ near $x_0$. Then $(x',\xi')$ parameterize the vectors near $(x_0,\theta_0)$.   
We define the Sobolev space $H^1(\mathcal{H})$ of functions constant along the flow, supported near the flow-out of $(x_0,\theta_0)$ as the $H^s$ norm in those coordinates w.r.t.\ the measure $\d\mu$. We can chose another such surface $H$ near $q_0$ with some fixed coordinates on it; the resulting norm will the be equivalent to that on $\mathcal{H}$.

\begin{proposition}\label{pr_map}  
With the notation and the assumptions above, for any $s\ge0$, the operators 
\begin{align}\label{map1}
X : H^s_0(V)& \longrightarrow H^{s+1/2}(\mathcal{H}),\\\label{map2}
X^*X : H^s_0(V)&\longrightarrow H^{s+1}(V)
\end{align}
are bounded. 
\end{proposition}

\begin{proof}
Recall first that the weight $\kappa$ localizes in a small neighborhood of $(\gamma_0,\dot\gamma_0)$. Let first $f$ has small enough support in a set that we will call $M_0$.  Then $M_0$ will be a simple manifold if small enough. 
 Then we can replace $H$ by another surface $H_0$ that lies in  $M_0$, and denote by $\mathcal{H}_0$ the corresponding $\mathcal{H}$ . This changes the original parameterization to a new one, that will give us an equivalent norm. 

Then, if $s$ is a half-integer,
\[
\|Xf\|_{H^{s+1/2}(\mathcal{H}_0)}^2 
\le C\sum_{|\alpha|\le 2s+1}\Big| \left(    \partial_{x',\xi'}^\alpha   Xf,Xf  \right)_{L^2(\mathcal{H}_0)}\Big| 
=   C\sum_{|\alpha|\le 2s+1}\Big| \left( X^*   \partial_{x',\xi'}^\alpha   Xf,f  \right)_{L^2(\mathcal{H}_0)}\Big|.
\]
The term $\partial_{x',\xi'}^\alpha   Xf$ is a sum of weighted ray transforms of derivatives of $f$ up to order $|\alpha|$. Then $X^*\partial_{x',\xi'}^\alpha X$ is a \PDO\ of order $|\alpha|-1$ because $M_0$ is a simple manifold. That easily implies
\[
\|Xf\|_{H^{s+1/2}(\mathcal{H}_0)} \le C\|f\|_{H^s}.
\]
The case of general $s\ge0$ follows by interpolation, see, e.g., \cite[Sec~4.2]{Taylor-book1}.

To finish a proof, we cover $\gamma_0$ with open sets so that the closure of each one is a simple manifold.  Choose a finite subset and  a partition of unity $1=\sum\chi_j$ related to that. Then we apply the estimate above to each $X{\chi_j} f$ on the corresponding $\mathcal{H}_j$. We then have finitely many Sobolev norms that are equivalent, and in particular equivalent to the one on $\mathcal{H}$. This proves \r{map1}.

To prove the continuity of $X^*X$, we need to estimate the derivatives of $X^*X$.  We have that $\partial^\alpha X^*Xf$ is sum of  operators $X_{\kappa_\alpha}$ of the same kind but with possibly  different weights applied to derivatives of $Xf$ up to order $|\alpha|$, see \r{Xstar}. 
Let first $s=0$. 
For $f$, $h$ in $C_0^\infty(V)$, $|\beta|=1$, we have
\[
\Big|\left( f, X_{\kappa_\beta}^*\partial_{x',\xi'}^\beta X  h \right)_{L^2(V) }\Big| \le C \| X_{\kappa_\beta}f\|_{H^{1/2}} \|  X  h \|_{H^{1/2} }\le C
\|f\|_{L^2(V)}\|h\|_{L^2(V)}.
\]
In the last inequality, we used \r{map1} that we proved already. This proves \r{map2} for $s=0$. 

For $s\ge1$, integer, we can ``commute'' the derivative in  $\partial^\alpha X^*X$ with $X^*X$ by writing it as a finite sum of operators of the type $X^*_{\tilde \beta}X_{\beta}P_\beta f$, $|\beta|\le|\alpha|$, where $P_\beta$ are differential operators of order $\beta$. To this end, we first ``commute'' it with $X^*$, as above, and then with $X$. Then we apply \r{map2} with $s=0$. The case of general $s\ge0$ follows by interpolation. 
\end{proof}

\begin{remark}
We did not use the fold condition here. In fact, Proposition~\ref{pr_map} holds without any assumptions on the type of the conjugate points, as long as $V$ is contained in a small enough neighborhood of a fixed geodesic segment that extends to a larger one with both endpoints outside $V$. Note that proving the mapping properties of $X^*X$ based on its FIO characterization is not straightforward, and we would get the same conclusion under some assumptions only, for example that the canonical relation is a canonical graph; that is not always true. 
\end{remark}

\begin{remark}\label{remark_global}
A global version of Proposition~\ref{pr_map} can easily be derived by a partition of unity in the phase space. Let $(M,g)$ be a compact non-trapping Riemannian manifold with boundary. Let $M_1$ be another such manifold which interior includes $M$, and assume that $\partial M_1$ is strictly convex. Such $M_1$ always exists if $\partial M$ is strictly convex. Let $\partial_-SM_1$ denote the vectors with base point on $\bo$ pointing into $M_1$. Then we can parameterize all (directed) geodesics with points in $\partial_-SM_1$, that plays the role of $\mathcal{H}$ above. 
Then for $s\ge0$, 
\[
X : H_0^s(M) \longrightarrow H^{s+1/2}(\partial_-SM_1) , \quad 
X^* X : H_0^s(M) \longrightarrow H^{s+1}(M_1) 
\]
are bounded. 
\end{remark}

\subsection{General regular exponential maps}
Let now $\exp$ be a regular exponential map. 
As above, we split the $t$-integral in the second line below into two parts to get
\be{NStar}
\begin{split}
Nf(p) &= \int {\kappa^\sharp}(p,\theta) Xf(p,\theta)\,\d\sigma_p(\theta)\\
      & =  \int_{S_pM} \int  {\kappa^\sharp}(p,\theta) \kappa\left(\exp_p(t\theta), \dot \exp_p(t\theta)\right)f(\exp_p(t\theta)) \, \d t\, \d \sigma_p(\theta)\\     
      & = \int_{T_pM} W(p,v)f(\exp_p(v))\, \d\Vol(v),
\end{split}
\ee
where
\be{W2}
\begin{split}
W  &= |v|^{-n+1}\Big( {\kappa^\sharp}(p,v/|v|)\kappa \big(\exp_p(v), \dot \exp_p(v)/| v| \big) \\
& \qquad\quad \qquad +  {\kappa^\sharp}(p,-v/|v|)\kappa \big(\exp_p(v), -\dot \exp_p(v)/| v| \big)\Big). 
\end{split}
\ee

\begin{theorem}  \label{thm_PDO}
Let $\exp_p(v)$ satisfy (R1) and (R4) and assume that for any $(p,\theta)\in \supp {\kappa^\sharp}$, $t\theta$ is not a conjugate vector at $p$ for $t$ such that $\exp_p(t\theta)\in\supp f$. 
 Then $N$ is a classical \PDO\ of order $-1$ with principal symbol
\be{thm_PDO_eq1}
\sigma_p(N)(x,\xi)  = 2\pi \int_{S_xM}\delta(\xi(\theta))  ( {\kappa^\sharp}\kappa)(x,\theta)  \,\d\sigma_x(\theta),
\ee
where $\xi(\theta)=\xi_i\theta^j$, and $\delta$ is the Dirac delta function.

\end{theorem}

\begin{proof}
The theorem is essentially proved in Section~4 of \cite{FSU}, where the exponential map is related to a geodesic like family of curves. We will repeat the arguments there in this more general situation.

Notice first that it is enough to study small enough $|t|$. 
Fix local coordinates $x$ near $p_0$. 
By (R4), 
\[
\exp_x(t\theta)= x+tm(t,\theta;x), \quad m(0,\theta;x)=\theta,
\]
with a smooth function $m$ near $(0,\theta_0,p_0)$. Introduce new variables $(r,\omega)\in\R\times S_xM$ by 
\[
r= t|m(t,\theta;x)|, \quad \omega = m(t,\theta;x)/|m(t,\theta;x)|,
\]
where $|\cdot|$ is the norm in the metric $g(x)$. Then $(r,\omega)$ are polar coordinates for $\exp_x(t\theta)-x=r\omega$ with $r$ that can be negative, as well, i.e., 
\[
\exp_x(t\theta) = x+r\omega.
\]
The functions  $(r,\omega)$ are clearly smooth got $|t|\ll1$, and $x$ close to $p_0$. Let 
\[
J(t,\theta;x) = \det \d_{t,v}(r,\omega)
\]
be the Jacobi determinant of the map $(t,v)\mapsto (r,\omega)$. By (R4), $J|_{t=0}=1$, therefore that map is a local diffeomorphism from $(-\eps,\eps)\times S_xM$ to its image for $0<\eps\ll1$. It is not hard to see that  for $0<\eps\ll1$ it is also a global diffeomorphism, because it is clearly injective. Let $t=t(x,r,\omega)$, $\theta=\theta(x,r,\omega)$ be the inverse functions defined by that map. Then
\[
t=r+O(|r|), \quad \theta = \omega+O(|r|), \quad \dot\exp(t\theta) = \omega+ O(|r|).
\]
Assume that the weight $\kappa$ in \r{1.2} vanishes for $p$ outside some small neighborhood of $p_0$. Then after a change of variables, we get
\[
Nf(x) = \int_{S_xM}\int A(x,r,\omega) f(x+r\omega) \, \d r\d\sigma_x(\omega),
\]
where
\[
A(x,r,\omega)=  {\kappa^\sharp}(x,\theta(x,r,\omega)) \kappa( x+r\omega,\omega+rO(1))  J^{-1}(x,r,\omega)
\]
with $J$ as before, but written in the variables $(x,r,\omega)$. By \cite[Lemma~2]{FSU}, $N$ is a classical \PDO\  with a principal symbol
\be{7.2}
2\pi\int_{S_xM}\delta(\xi(\omega)) A(x,0,\omega)\, \d\sigma_x(\omega) = 
2\pi\int_{S_xM}\delta(\xi(\omega))  {\kappa^\sharp}(x,\omega)\kappa(x,\omega)\, \d\sigma_x(\omega).
\ee
\end{proof}

\begin{remark}\label{remarkW} 
Formulas \r{Xstar2} and \r{NStar} are valid regardless of possible conjugate points. In our setup, the supports of $\kappa$, $\kappa^\sharp$ guarantee that $\exp_p(t\theta)$, for $(p,\theta)$ close to $(p_0,\theta_0)$ reaches a conjugate point for $t>0$ but not for $t<0$. Therefore, near the conjugate point $q$ of $p$, the second term on the r.h.s.\ of \r{W}, and \r{W2}, respectively, vanishes. 
\end{remark}

\section{The Schwartz kernel of $N$ near the conjugate locus $\Sigma$} \label{sec_diag}

We will introduce first three invariants. Let $F:M\to N$ be a smooth orientation preserving map between two orientable Riemannian manifolds $(M,g)$ and $(N,h)$. Then one defines $\det \d F$ invariantly by
\be{IF}
F^* (\d\Vol_N) =   (\det \d F)\,  \d\Vol_M,
\ee
see  also \cite[X.3]{Lang_manifolds}. In local coordinates, 
\be{dF}
\det \d F (x)= \sqrt{\frac{\det h(F(x))}{\det g(x)}}\det \frac{\partial F(x)}{\partial x} .
\ee

We choose an orientation of $S(p_0)$ near $v_0$, as a surface in $T_{p_0}M$  by choosing a unit normal field so that the derivative of $\det \d \exp_{p_0}(v)$ along it is positive on $S(p)$. Then we extend this orientation to $S(p)$ for $p$ close to $p_0$ by continuity. On Figure~\ref{fig:caustics_multidim}, the positive side is the one below $S(p)$, if $v$ is the first conjugate vector along the geodesic through $(p,v)$. Then we choose an orientation of $\Sigma(p)$ so that the positive side is that in the range of $\exp_p$. On Figure~\ref{fig:caustics_multidim}, the positive side is to the left of $\Sigma(p)$. 
The so chosen orientations conform with the signs of $\xi^n$ and $y^n$ in the normal form \r{2.1}.

Next we synchronize the orientations of $T_pM$ and $M$ near $q$ by postulating that $\exp_p$ is an orientation preserving map from the positive side of $S(p)$, as described above, to the positive side of $\Sigma(p)$.

For each $p\in M$, the transformation laws in $TT_pM$ under coordinate changes on the base show that $T_pM$ has the natural structure of a Riemannian manifold with the constant metric $g(p)$. 
Then one can define $\det \d\exp_p$ invariantly as above. Let $\d\Vol_p$ be the volume form in $T_pM$, and let $\d\Vol$ be the volume form in $M$. Then $\det \d\exp_p$ is defined invariantly by
\be{I1}
\exp_p^* \d\Vol =   \left(\det \d\exp_p\right)  \d\Vol_p.
\ee
In local coordinates, 
\[
\det \d\exp_p = \sqrt{\frac{\det g( \exp_p v  )}{\det g(p)}}\det \frac{\partial}{\partial v} \exp_p (v),
\]
where, with some abuse of notation, $g(p)$ is the metric $g$ in fixed coordinates near a fixed $p_0$, and $g( \exp_p v  )$ is the metric $g$ in a possibly different system of fixed coordinates near  $q_0 =\exp_{p_0}v_0$. Set
\be{inv1}
A(p,v) := |\d \det \d \exp_p(v)|.
\ee
Since $\det \d \exp_p(v)$ is a defining function for $S(p)$, its differential  is conormal to it. 
By the fold condition, $A\not=0$. One can check directly that $A$ is invariantly defined on $\Sigma$.  

By \r{2.1n}, for $(p,v)\in S$, the differential of $\exp_p$ maps isomorphically $T_vS(p)$ (equipped with the metric on that plane induced by $g(p)$) into $T_q\Sigma$, with the induced metric. 
Let $D$ be the  determinant of $\exp_p|_{S(p)}$, i.e.,
\be{inv2}
D := \det \left( \d \exp_p |_{T_vS(p)} \right),
\ee
defined invariantly by \r{IF}. We synchronize the orientations of $S(p)$ and $\Sigma(p)$ so that $D>0$.

We express next the weight $W(p,v)$ restricted to $S$ in terms of the variables $(p,q)$. For $(p,q)\in \Sigma$, $v=\exp^{-1}_p(q)$, where we inverted $\exp_p$ restricted to $S$. Let $w=w(p,q)$ be defined as in \r{qw} with $v$ as above. 
Then we set, see also \r{W2}, and Remark~\ref{remarkW},
\begin{equation}\label{WS}
W_\Sigma(p,q) := W\big((p,\exp^{-1}_p(q))|_\Sigma
 = |v|^{1-n} {\kappa^\sharp}(p,v/|v|)\kappa(q,-w/|v|)
\end{equation}

For $p$ close to $p_0$, $\Sigma(p)$ divides $M$ in a neighborhood of $q_0$ into two parts: one of them  is in the range of $\exp_p(v)$ for $v $ near $v_0$, that is the positive one w.r.t.\ the chosen orientation;  the other is not.  Let $z'(p,q)$ be the distance from $q$ to $\Sigma(p)$ with a positive sign in the first region, and with a negative sign in the second one. Then for a fixed $p$, $z'=z'(p,q)$ is a normal coordinate to $\Sigma(p)$ depending smoothly on $p$, and $\Sigma$ is given locally by $z'=0$. Then $z'$ is a defining function for $\Sigma$, i.e., $\Sigma=\{z'=0\}$ and $\d_{p,q}z'\not=0$ because $\d_qz'\not=0$. 
Let $z''=z''(p,q)\in \R^{2n-1}$ be such that its differential restricted to $T\Sigma$ is an isomorphism at $(p_0,q_0)$.  Since $\d z''$ and $\d z'$ are linearly independent, $z=z(z',z'')$ are coordinates near $(p_0,q_0)$. One way to construct $z''$ is the following. Choose $(z_{n+1},\dots,z_{2n})$, depending on $p$ only, to be local coordinates for $p$, and to choose $(z',z_2,\dots,z_{n})$, depending on $p$ and $q$, to be semi-geodesic  coordinates of $q$ near $\Sigma(p)$. 

The next theorem shows that near $\Sigma$, the operator $N$ has a singular but integrable kernel with a conormal singularity of the type $1/\sqrt{z'}$. 

\begin{theorem}\label{thm_kernel} Near $\Sigma(p)$, 
the Schwartz kernel $N(p,q)$ of $N$ (with respect to the volume measure)  near $(p_0,q_0)$ is of the form
\be{thm_k1}
N = W_\Sigma \frac{\sqrt2}{\sqrt{AD z'}}( 1+ \sqrt{z'} R(\sqrt{z'},z'')),
\ee
where $W_\Sigma= W_\Sigma(z'')$, $A=A(z'')$, $D=D(z'')$,  and  $R$ is a smooth function.
\end{theorem}

\begin{proof}
We start with the representation \eqref{NStar}. We will make the change of variables $y=\exp_p(v)$ for $(p,v)$ close to $(p_0,v_0)$ as always. Then $y$ will be on the positive side of $\Sigma(p)$, and the exponential map is 2-to-1 there. We split the integration in \r{NStar} in two parts: one, where $v$ is on the positive side of $S(p)$, that we call $N_+f$, and the other one we denote by $N_-f$. Then
\be{g1}
\begin{split}
N_\pm f(p) &= \int_{S_pM} \int W f(y) \big(\det \,\d \exp_p^\pm (v)\big)^{-1}  \,\d\Vol(y),
\end{split}
\ee
where $W$ is as in \r{WS} but not restricted to $\Sigma$, and $(\exp_p^\pm)^{-1}$ there is the corresponding inverse in each of the two cases. 

To prove the theorem, we need to analyze the singularity of the Jacobian determinant $\det \,\d \exp_p(v)$ near $\Sigma(p)$. It is enough to do this at $(p_0,v_0)$. 

Let $y=(y',y^n)$ be semi-geodesic coordinates near $\Sigma(q_0)$, $q_0=\exp_{p_0} (v_0)$, and let $y_0$ correspond to $q_0$. We assume that 
$y^n>0$ on the positive side of $\Sigma(p)$. In other words, $y^n=z'(p_0,q)$.

We have
\[
 \d\Vol(y)   =    \det \big(\d_v \exp_p(v)\big)\, \d \Vol(v)
\]
The form on the left can be written as $\d\Vol_{\Sigma(p)}(y')\, \d y^n$; while the one on the right, restricted to $S(p)$, equals $\d \Vol_{S(p)}(v')\,\d v^n$ in boundary normal coordinates to $S(p)$, where $v^n>0$ gives the positive side of $S(p)$.  On the other hand, by \r{inv2}, 
\[
 \d\Vol_{\Sigma(p)}(y')   =   D\,  \d \Vol_{S(p)}(v').
\]
We therefore get 
\[
D\, \d y^n =\det \big(\d \exp_p(v)\big)\,\d v^n. 
\]
By the definition of $A$, we have 
\be{4.02}
\det\d_v \exp_p(v)= Av^n(1+O(v^n)).
\ee 
Therefore,
\[
D\, \d y^n = A(1+O(v^n))\, v^n\d v^n.
\]
Since $y^n=0$ for $v^n=0$, we get
\[
y^n = (v^n)^2\frac{A  }{2D} (1+O(v^n)).
\]
Solve this for $v^n$ and plug into \r{4.02} to get
\be{4.3}
\det \d \exp_p(v)= \pm \sqrt{2ADy^n}\left(1+O_\pm\big(\sqrt{y^n}\big)\right).
\ee
Here $O_\pm\big(\sqrt{y^n}\big)$ denotes a smooth function of $\sqrt{y^n}$ near the origin with coefficients smooth in $y'$, that vanishes at $y^n=0$. 
The positive/negative sign corresponds to $v$ belonging to the positive/negative side of $S(p)$. 
By \r{g1}, 
\be{4.4}
N_\pm f(p) = \int W f(y) \frac{1}{\sqrt{2ADy^n}}\left(1+O_\pm\big(\sqrt{y^n}\big)\right)   \d\Vol(y).
\ee
We replace $A_0$, $D_0$ in \r{4.4} by their values at $y^n =0$; the error will then just replace the remainder term above by another one of the same type.  Similarly, $W=W(p,v)$, where $\exp_p(v)=q$. Solving the latter for $v=v(p,q)$ provides a function having a finite Taylor expansion in powers of $\sqrt{y^n}$ of any order, with smooth coefficients. The leading term is what we denoted by $W_\Sigma$ that is a smooth function on $\Sigma$.

With the aid of  \r{dF}, it is easy to see that \r{4.4} is a coordinate representation of the formula \r{thm_k1} at the so fixed $p$. When $p$ varies near $p_0$, it is enough to notice that since we already wrote the integral in invariant form, $y^n$ then becomes the function $z'(p,q)$ introduced above. For $z''$ we then have $z''(p,q)= (x(p),y'(p,q))$. Finally, we note that another choice of $z''$ so that $(z',z'')$ are coordinates would preserve \r{thm_k1} with a possibly different $R$. 
\end{proof}

\section{$N$ as a Fourier Integral Operator. Proof of Theorem~\ref{thm_main}}

We are ready to finish the proof of Theorem~\ref{thm_main}. By Theorem~\ref{thm_kernel}, near $\Sigma$, the Schwartz kernel of $N$ has a conormal singularity at $\Sigma$, supported on one side of it, that admits a singular expansion in powers of $\sqrt{z'_+}$, with a leading singularity $1/\sqrt{z'_+}$. The Fourier transform of the latter is 
\be{FTz}
\sqrt{\pi}e^{-\i\pi/4}(\zeta_+^{-1/2}+ \i \zeta_-^{-1/2})
\ee
where $\zeta_+=\max(\zeta,0)$, $\zeta_-=(-\zeta)_+$. The singularity near $\zeta=0$ can be cut off, and we then get a symbol of order $-1/2$, depending smoothly on the other $2n-1$ variables. Therefore, near $\Sigma$, the kernel of $N$ belongs to the conformal class $I^{-n/2}(M\times M,\Sigma;\mathbf{C})$, see e.g., \cite[18.2]{Hormander3}. It is elliptic when $ {\kappa^\sharp}(p_0,\theta_0) {\kappa^\sharp}(q_0,-w_0)\not=0$ by \eqref{W2}, \r{WS}. Therefore, the kernel of $N$ near $\Sigma$ is a kernel of an FIO associated to the Lagrangian $T^*\Sigma$. Moreover, the amplitude of the conormal singularity at $\Sigma$ is in the class $S_{\rm phg}^{-1/2,1/2}$ (polyhomogeneous of order $-1/2$, having an asymptotic expansion in integer   powers of $|\zeta|^{1/2}$), see also \r{sin3a} and \r{sin3a1}.

\section{The two dimensional case}
\begin{theorem}\label{thm_2D}
Let $\dim M=2$. Assume that (R1) -- (R5) are fulfilled. Then $\mathcal{N}^*\Sigma\setminus 0$, near $(p_0,\xi_0,q_0,\eta_0)$, is the graph of a local diffeomorphism $T^*M\setminus 0 \in (p,\xi) \mapsto (q,\eta)\in  T^*M\setminus 0$, homogeneous of order one in its second variable (a canonical graph). 
\end{theorem}

\begin{proof}
For $(p,\xi)$ near $(p_0,\xi_0)$, there are exactly two smooth maps that map $\xi$ to a unit normal vector. We choose the one that maps $\xi_0$ to $v_0/|v_0|$. Then we map the latter to $v\in S(p)$. Since the radial ray through $v$ is transversal to $S(p)$, that map is smooth. Knowing $v$, then we can express $q=\exp_p(v)\in \Sigma(p)$ and $w = -\dot\exp_p(v)$ as smooth functions of $(p,\xi)$ as well. Then in local coordinates, $\eta = \xi_i\partial \exp_q^i(w)/\partial q$, see \r{2.5cn}, that in particular proves the homogeneity. 

By (R5), this map is invertible. 
\end{proof}

\begin{figure}[h] 
  \centering
  \includegraphics[bb=0 0 534 384,width=3.66in,height=2.63in,keepaspectratio]{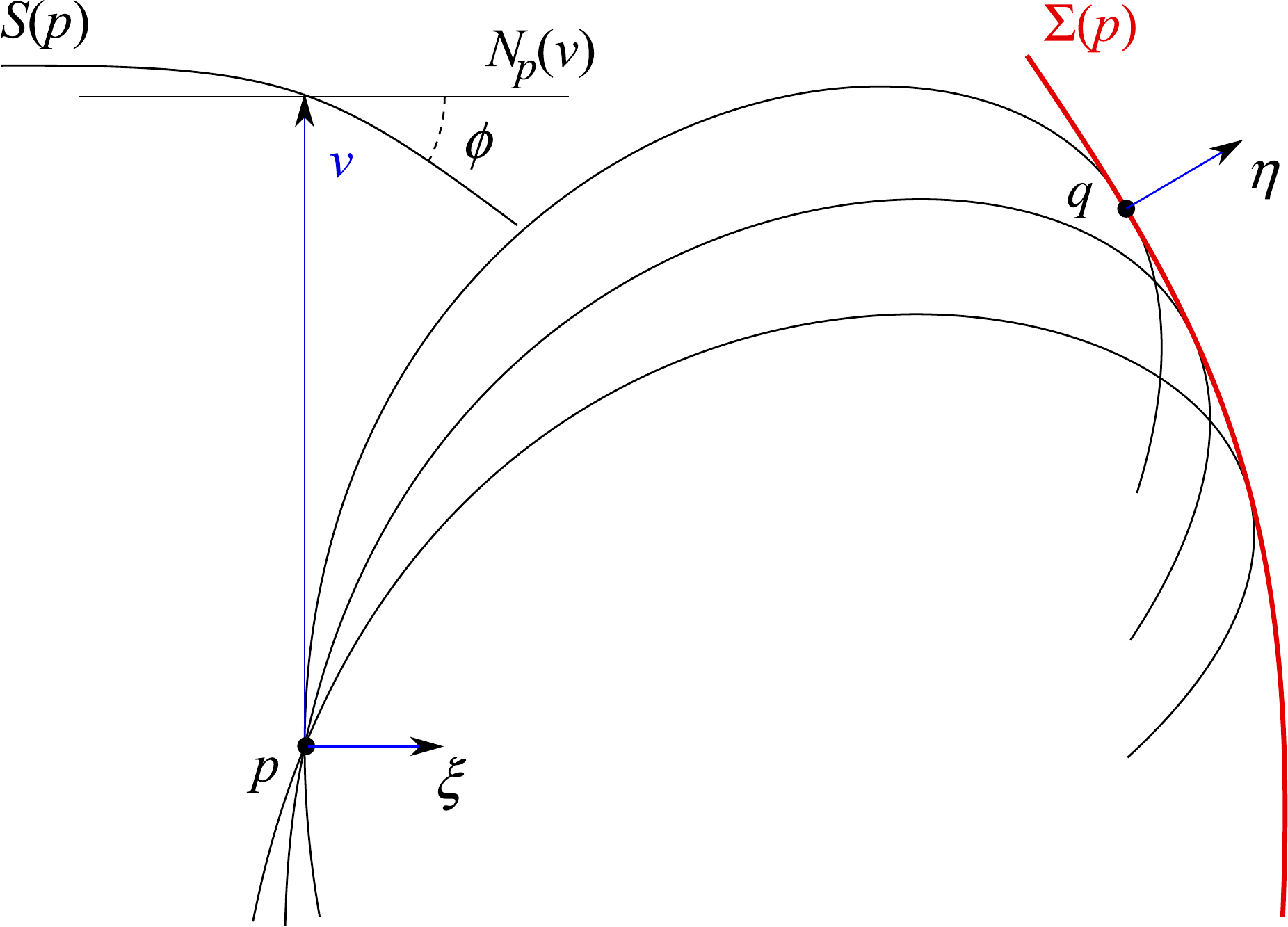}
  \caption{The 2D case}
  \label{fig_2}
\end{figure}

The principal symbol of $X^*X$ in the geodesics case, see Theorem~\ref{thm_Duke}, and \r{a2}, is given by 
\be{pr2D}
\sigma_{\rm p}(X^*X)(x,\xi) = 2\pi|\kappa(x,\xi^\perp/|\xi^\perp|)|^2,
\ee
where $\xi^\perp$ is a continuous choice of a vector field normal to $\xi$ and of the same length so that at $p=p_0$, $\xi_0^\perp/|\xi_0^\perp|=\theta_0$, $-\xi_0^\perp/|-\xi_0^\perp|=\theta_0$; therefore, the sign of the angle of rotation is different near $\xi_0$ and near $-\xi_0$. Notice that \r{a2} in the two dimensional case is a sum of two terms but we assumed that $\kappa$ is supported neat $(p_0,\theta_0)$, therefore only one of the terms is non-trivial. A similar remark applies to \r{thm_PDO_eq1}.  

Theorem~\ref{thm_kernel} takes the following form in two dimensions, in the Riemannian case.

\begin{corollary}  \label{cor_2D} Let $n=2$ and let $\exp$ be the exponential map of a Riemannian metric. 
With the notation of Theorem~\ref{thm_kernel},  we then have
\be{thm_k12d}
N = W_\Sigma \frac{\sqrt2}{\sqrt{B z'}}( 1+ \sqrt{z'} R(\sqrt{z'},z'')),
\ee
where 
\[
B  = \Big|\frac{\d}{\d N}  \det \,\d\, \exp_p(v)\Big|
\]
is evaluated at $v\in S(p)$ such that $q=\exp_p(v)$, and $\d/ \d N$ stands for the derivative in the direction of $N_p(v)$. 
\end{corollary}

\begin{proof}
Note first that $B\not=0$ by the fold condition. Let $\phi$ be the (acute) angle between $S(p)$ and $N_p(v)$ at $v$. Since $N_p(v)$ is orthogonal to the radial ray at $v$, we can introduce an orthonormal coordinate system at $v$ with the first coordinate vector being $v/|v|$, and the second one: the positively oriented  unit vector along $N_p(v)$, that we call $\xi$. 
Let us parallel transport this frame along the geodesic $\gamma_{p,v}$; and invert the direction of the tangent vector to conform with our choice of $w$ at $q$. In particular, this introduces a similar coordinate system near the corresponding vector $w$ at $q$ in the conjugate locus.  
In these coordinates then
\be{DE}
\d \exp_p(v) = \begin{pmatrix}-1 &0\\0 &j/|v|\end{pmatrix},
\ee
where $j$ is uniquely determined by $J(t)=j(t)\Xi(t)$, where $J(t)$ is the Jacobi field with $J(0)=0$, $J'(0)=\xi$, and  $\Xi(t)$ is the parallel transport of $\xi$, compare that with \r{JJ}. The extra factor $1/|v|$ comes from the fact that we normalize $v$ now in our basis, so that the result would be the Jacobian determinant. 
Then the Jacobi determinant $\det\d \exp_p(v) $ is given by $-j/|v|$.  In particular, for $(p,v)\in S$ we have $\d \exp_p(v)= \mbox{diag}(-1,0)$. Note that $j$ depends on $v$ as well, therefore its differential that essentially gives  $\d \det\d \exp_p(v) $ depends on the properties of the Jacobi field under a variation of the geodesic.

Now, it easily follows from the definition \r{inv2} of $D$ that 
\[
D = \sin\phi.
\]
On the other hand, $\d \det \d \exp_p(v)$ is conormal to $S(p)$, therefore, the derivative of $\det \d \exp_p(v)$ in the direction of $N_p(v)$ satisfies
\[
\Big|\frac{\d}{\d N}  \det \,\d\, \exp_p(v)\Big| = |\d \det \d \exp_p(v)|\sin\phi =A\sin\phi=AD.
\]
\end{proof}

\section{Resolving the singularities in the geodesic case}\label{sec_9}
Let, as before, $(p_0,q_0)$ be a pair of fold conjugate points along $\gamma_0$, and $X$ be the ray transform with a weight that localizes near $\gamma_0$. We want to see whether we can resolve the singularities of $f$ near $p_0$ and near $q_0$  knowing that $Xf\in C^\infty$, and more generally, whether we can invert $X$ microlocally.  Assume for simplicity that $p_0\not=q_0$.  

We will restrict ourselves to the geodesic case only but the same analysis holds without changes to the case of magnetic geodesics as well. We avoid the formal introduction of magnetic geodesics for simplicity of the exposition.  
Assume also that 
\be{ellipt}
\kappa(p,\theta)\kappa(q,-w/|w|)\not=0, \quad \mbox{for $(p,\theta)\in \mathcal{U}_0$},
\ee
where $(q,w)$ are given by \r{qw}, and $\mathcal{U}\Supset\mathcal{U}_0\ni (p_0,\theta_0)$. 
This guarantees the microlocal ellipticity of the \PDO\ $A$ near $\mathcal{N}^*(p_0,v_0)$ and $\mathcal{N}^*(q_0,w_0)$  
in Theorem~\ref{thm_main}, see Theorem~\ref{thm_Duke}.

\subsection{Sketch of the results} 
We explain the results before first in an informal way. As we pointed out in the Introduction, $Xf(\gamma)$ for geodesics near $\gamma_0$ can only provide information for $\WF(f)$ near $\mathcal{N}^*\gamma_0$, and does not ``see'' the other singularities. The analysis below based on Theorem~\ref{thm_main}, shows that on a principal symbol level, the operator $|D|^{1/2}F$ behaves as a Radon type of transform on the curves (when $n=2$) or the surfaces (when $n\ge3$) $\Sigma(p)$. Similarly, its adjoint behaves as a Radon transform on the curves/surfaces $\Sigma(q)$. Therefore, there are two geometric objects that can detect singularities at $p_0$ conormal to $v_0$: the geodesic $\gamma_0= \gamma_{p_0,v_0}$ (and those close to it) and the conjugate locus $\Sigma(q_0)$ through $p_0$ (and those corresponding to perturbations of $v_0$). We refer to Figure~\ref{fig:caustics4}.

When $n=2$, the information coming from integrals along the two curves (and their neighborhoods) may in principle cancel; and we show in  Theorem~\ref{thm_cancel} that this actually happens, at least to order one. When $n\ge3$, the Radon transform over $\Sigma(q)\ni p$ competes with the geodesic transform over  geodesics through $p$. Depending on the properties of that Radon transform,  the information that we get for $\pm \xi_0$ may or may not cancel because $\xi_0$ is conormal both to $\gamma_0$ and $\Sigma(q_0)$. On the other hand, for any other $\xi_1$ conormal to $v_0$ but not parallel to $\xi_0$, the geodesic $\gamma_0$ (and those close to it) can detect whether it is in $\WF(f)$ but the Radon transform restricted to small perturbations of $v_0$ (and therefore of $q_0$) will not. Thus, we can invert $N$ microlocally at such $(p_0,\xi_1)$. 

Now, when $n\ge3$, we may try to invert $N$ even at $\xi_0$ by choosing $v$'s close to $v_0$ but normal to $\xi_0$. If $\xi_0$ happens not to be conormal to the corresponding conjugate locus $\Sigma(q(p_0,v))$ at $p_0$, we can just use the argument above with the new $v$. In particular, if the map \r{LG} is a local diffeomorphism, this can be done. 

This suggests the following sufficient condition for inverting $N$ at $(p_0,\xi_1)$:
\be{Scon}
\text{$\exists \theta_1\in S_{p_0}M$, so that 
$\kappa(p_0,\theta_1)\not=0$, $\xi_1(\theta_1)=0$, and $\xi_1$ is not conormal to $\Sigma(q(p_0,\theta_1))$ at $p_0$.}
\ee
Above, $\Sigma(q(p_0,\theta_1))$ is the conjugate locus to the point $q$ that is conjugate to $p_0$ along $\gamma_{p_0,\theta_1}$. We normally denote that point by $q(p_0,v_1)$, where $v_1\in S(p_0)$ has the same direction as $\theta_1$.  

In case of the geodesic transform, one could formulate \r{Scon} in terms of the map \r{LG} as follows:
\be{Scon2}
\text{$\exists v_1\in S(p_0)$, so that $\kappa(p_0,v_1/|v_1|)\not=0$, $\xi_1(v_1)=0$, and $\xi_1$ is not the image of $v_1$ under the map \r{LG} at $p_0$.}
\ee

In Section~\ref{sec_mag}, we present an example where \r{LG} is a local diffeomorphism, therefore \r{Scon} holds. In Section~\ref{sec_prod} we present another example, where \r{Scon} fails.

\begin{figure}[tbp] 
  \centering
  \includegraphics[bb=9 0 507 321,width=3.66in,height=2.36in,keepaspectratio]{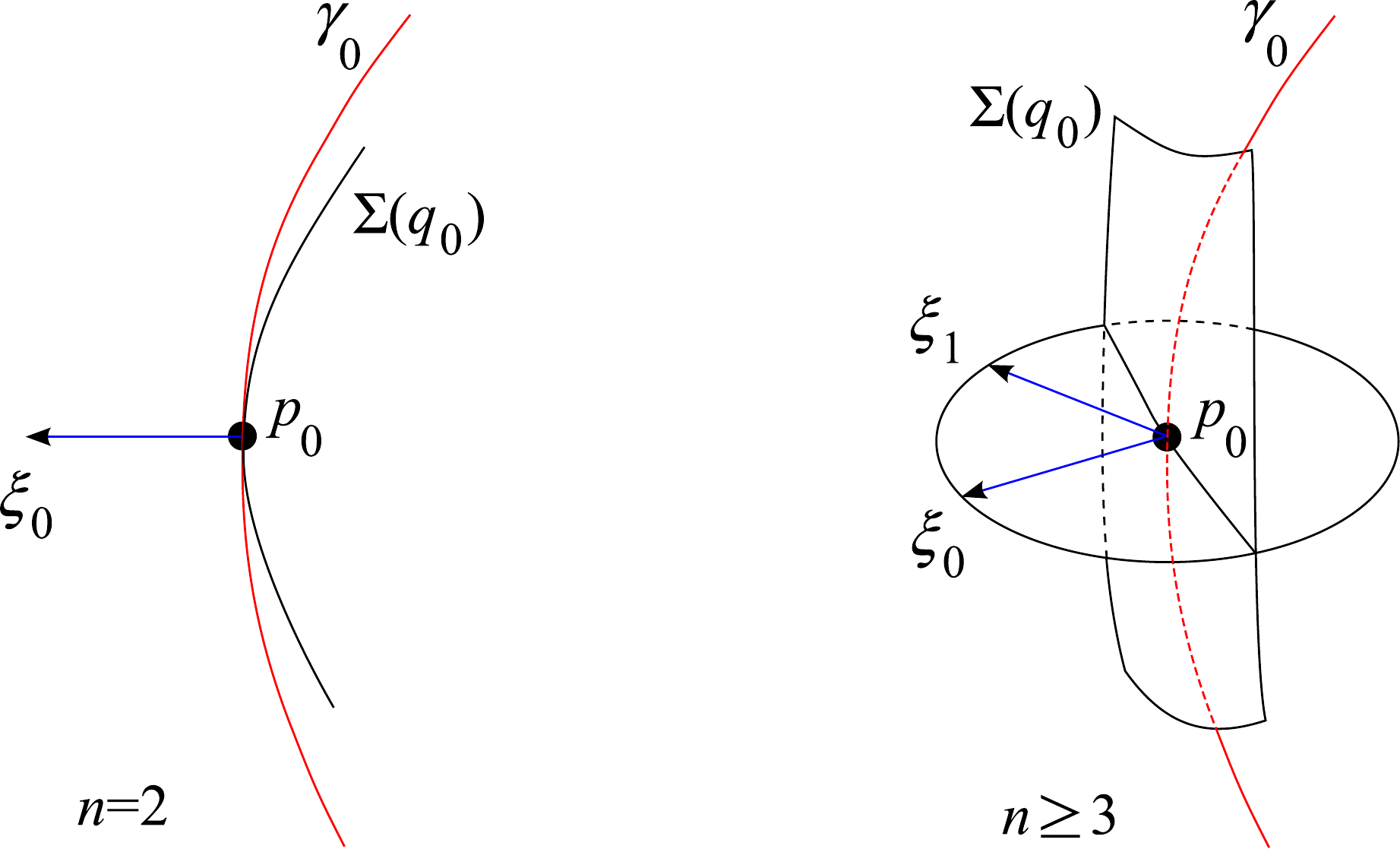}
  \caption{}
  \label{fig:caustics4}
\end{figure}

\bigskip

\subsection{Recovery of singularities in all dimensions} 
We proceed next with analysis of the recovery of singularities.

Let $\chi_{1,2}$ be smooth functions on $M$ that localize near $p_0$, and $q_0$, respectively, i.e., $\supp\chi_1\subset U_1$, $\supp\chi_2\subset U_2$, where $U_{1,2}$ are small enough neighborhoods of $p$ and $q$, respectively.  Assume that $\chi_1$, $\chi_2$ equal $1$ in smaller neighborhoods of $p_0$, $q_0$, where $f_1$, $f_2$ are supported. 
Then $f:= f_1+f_2$ is supported in $U_1\cup U_2$ and we can write
\be{N1}
\chi_1N f =  A_1 f_1+ F_{12} f_2,
\ee
where $A_1 = \chi_1 N\chi_1$ is a \PDO\ by Theorem~\ref{thm_PDO}, while $F_{12} = \chi_1 N\chi_2$ is the  FIO that we denoted by $F$ in Theorem~\ref{thm_main}.
By  (R5), we can do the same thing near $q_0$ to get
\be{N2}
\chi_2 N f = A_2 f_2+ F_{21}f_1,
\ee
where $A_2 = \chi_2 N\chi_2$,  $F_{21} = \chi_2N\chi_1$. It follows immediately that $F_{21}=F_{12}^*$. Recall that $F_{12}=F$ in the notation of Theorem~\ref{thm_main}. 
Assuming $X^*Xf\in C^\infty$, we get
\be{A1A2}
A_1 f_1+ F f_2\in C^\infty, \quad A_2 f_2+ F^*f_1\in C^\infty.
\ee
Solve the first equation for $f_2$, plug into the second one to get
\be{N3}
\big(\Id - A_2^{-1}F^*A_1^{-1}F\big)f_2\in C^\infty \quad \text{near $(q_0,\pm\eta_0) $ },
\ee
where $A_1^{-1}$,  $A_2^{-1}$, denote parametrices of $A_1$, $A_2$ near $(p_0,\pm \xi_0)$, and $ (q_0,\pm\eta_0) $, respectively. 
The operator in the parentheses is a \PDO\ of order $0$ if the canonical relation is a graph, that is true in particular when $n=2$, by Theorem~\ref{thm_2D}. In that case, if $\Id- A_2^{-1}F^*A_1^{-1}F$ is an elliptic (as a \PDO\  of order 0), near  $(q_0,\pm\eta_0),$ then we can recover the singularities. Without the canonical graph assumption, if it is hypoelliptic, then we still can. 

Another way to express the arguments above is the following. Since $\chi_{1,2}$ together with  $\kappa$ restrict to conic neighborhoods of $(p_0\pm\xi_0)$, and $(q_0\pm\eta_0)$, respectively, and $A_{1,2}$, $F$, $F^*$ have canonical relations of graph type that preserve the union of those neighborhoods, we may think of $f=f_1+f_2$ as a vector $f=(f_1,f_2)$, and then 
\be{Fmatrix}
F =
\begin{pmatrix}
 A_1&F\\F^*&A_2
\end{pmatrix}.
\ee
The operator $\Id - A_2^{-1}F^*A_1^{-1}F$ can be considered then as the ``determinant'' of $F$, up to elliptic factors.

\begin{theorem}\label{thm_sing}
Let the canonical relation of $F$ be a canonical graph. 
With the assumptions and the notation above, if the zeroth order \PDO\ 
\be{main_question}
\Id -A_2^{-1}F^*A_1^{-1}F
\ee
is elliptic in a conic neighborhood of  $(q_0,\pm\eta_0)$, then $Xf\in C^\infty$ near $(p_0,\theta_0)$ (or more generally, $Nf\in C^\infty$ near $p_0$ and $q_0$) implies $f\in C^\infty$. 
\end{theorem}

In the geodesic case in two dimensions, the principal symbol of $A_2^{-1}F^*A_1^{-1}F$ is always $1$, see the Proposition~\ref{pr_2D_symbol} below.

When $n\ge3$ and $F$ is of graph type, then $A_2^{-1}F^*A_1^{-1}F$ is of negative order, therefore we can resolve the singularities.

\begin{corollary}\label{cor_sing}
Let $n\ge3$ and assume that the canonical relation of $F$ is a canonical graph. Then the conclusions of Theorem~\ref{thm_sing} hold, i.e., $Xf\in C^\infty$ near $(p_0,\theta_0)$ (or more generally, $Nf\in C^\infty$ near $p_0$, $q_0$) implies $f\in C^\infty$. 
\end{corollary}

\begin{proof}
In this case, $A_1^{-1}F$ is an FIO of order $1-n/2$ with the same canonical relation is $F$. Similarly $A_2^{-1}F^*$ is an FIO of order $1-n/2$ with a canonical relation that is a graph of the inverse canonical map. Their composition is therefore a \PDO\ of order $2-n<0$. Its principal symbol as a \PDO\ of order $0$ is zero. The corollary now follows from Theorem~\ref{thm_sing}. 
\end{proof}

In Section~\ref{sec_mag}, we give an example where the assumptions of the corollary hold. Note that those assumptions are stable under small perturbations of the dynamical system. 

When the graph condition does not hold, the analysis is harder. Then \r{LG} is not a local diffeomorphism. If its range is a lower dimension submanifold, for example, we can at least recover the conormal singularities to $\theta_0$ away from it, as the corollary below implies. Note that below, (b) implies (a). Also, \r{ellipt} is not needed; only ellipticity of $\kappa$ at $(p_0,\theta_0)$ suffices.

\begin{corollary}\label{cor_2} Let $Xf\in C^\infty$ for $\gamma$ near $\gamma_0$. Then

(a) If $\xi_1\in T_{p_0}M\setminus 0$ is conormal to $v_0$ but not conormal to $\Sigma(q_0)$ (not parallel to $\xi_0$), then 
\[
(p_0,\xi_1)\not\in \WF(f). 
\]

(b) The same conclusion holds if condition \r{Scon} or the equivalent \r{Scon2} is fulfilled.
\end{corollary}

\begin{proof}
Note first that $A_1$ is elliptic at $(p_0,\zeta)$ by \r{ellipt} and Theorem~\ref{thm_Duke}(b). By the first relation in \r{A1A2}, $(p_0,\xi_1)\in \WF(f_1)$ if and only if $(p_0,\xi_1)\in \WF(Ff_2)$. To analyze the latter, we will use the relation $\WF(Ff_2)\subset \WF'(F)\circ \WF(f_2)$, see \cite[Thm~8.5.5]{Hormander1}. Note also that in the notation in \cite[Thm~8.5.5]{Hormander1}, $\WF(F)_X$ is empty. By Theorem~\ref{thm_kernel}, $\WF'(F)$ consists of those points in the canonical relation $\mathcal{C}$, see \r{C}, for which the conormal singularity in \r{thm_k1} is not canceled by a zero weight. 

Now, let $\xi_1$ be as in (a). Since $\xi_1$ is separated by $\pm \xi_0$ by a conic neighborhood, one can choose a weight $\chi$ on $SM$ that is constant along the geodesic flow, non-zero at $(p_0,\theta_0)$ and supported in a flow-out of a neighborhood $\mathcal{V}$ of it small enough such that the conormals to the corresponding conjugate loci at $p_0$ stay away from a neighborhood of $\xi_1$. In the geodesics case, the condition is that the map \r{LG} restricted to $\mathcal{V}$, does not intersect a chosen small enough conic neighborhood of $\pm \xi_0$. This can always be done by continuity arguments.   Then left projection of $\WF'(F)$ will not be singular at $(p_0,\xi_1)$, and therefore, $Ff_2$ will have the same property regardless of the singularities of $f_2$. 

Statement (b) follows from (a) by varying $v$ near $v_0$ in directions normal to $\xi_1$. 
\end{proof}

\newcommand{\m}{r}
\newcommand{\M}{R}

\subsection{Calculating the principal  symbol of \r{main_question} in case of Riemannian surfaces.} 
Let $\exp$ be the exponential map of $g$, and let $n\ge2$. We will take $n=2$ later. 
Recall that the leading singularity of the kernel of $N$ near $\Sigma$ is of the type $(z'_+)^{-1/2}$, by Theorem~\ref{thm_kernel}. We will compose $F$ with a certain \PDO\  $\M$ so that this singularity becomes of the type $\delta(z')$. Then modulo lower order terms, $F\M f(p)$ will be a weighted Radon transform over the surface $\Sigma(p)$. In 2D, that will be an X-ray type of transform. We are only interested in this composition acting on distributions with wave front sets in a small conic neighborhood $\mathcal{W}$ of $(q_0,\pm \eta_0)$. 

The Fourier transform of $(z'_+)^{-1/2}$ is given by \r{FTz}. Its reciprocal is 
$$
\pi^{-1/2}e^{\i \pi/4}\left( h(\zeta)\zeta^{1/2}-\i h(-\zeta)(-\zeta)^{1/2}\right) 
= \pi^{-1/2}e^{\i \pi/4}\big( h(\zeta)-\i h(-\zeta) \big)|\zeta|^{1/2},
$$
where $h$ is the Heaviside function, and $|\zeta|$ is the norm in $T^*_yM$.  
We fix $p$ near $p_0$ and local coordinates $x=x(p)$ there, and we work in semi-geodesic  coordinates $y=y(p,q)$ near $q_0$ normal to $\Sigma(p)$ oriented as in section~\ref{sec_diag}. Let $x$ denote local coordinates near $q_0$. 
Let $\M$ be a properly supported  \PDO\ of order $1/2$ with principal symbol, equal to
\be{sin_q}
\m(y,\eta) = \pi^{-1/2}e^{\i \pi/4}\big( h(\eta_n)-\i h(-\eta_n) \big)|\eta|^{1/2} r_0(y,\eta),
\ee
in $\mathcal{W}$, outside some neighborhood of the zero section, where $r_0$ is a homogeneous symbol of order $0$, an even function of  $\eta$.  
 Note that
\be{sin2}
|\m|^2= \pi^{-1}|\eta| r_0^2.
\ee
The appearance of the Heaviside function here can be explained by the fact that $N^*\Sigma$ has two connected components: near $(p_0,q_0,-\xi_0,\eta_0)$ and near $(p_0,q_0,\xi_0,-\eta_0)$; and the constants needs to be chosen differently in each component. 

We start with computing the composition 
\be{sin1}
F\M.
\ee

Since the kernel of \r{sin1} is the transpose of that of $\M F'$, we will compute the latter; and we only need those singularities that belong to $\mathcal{W}$. Denote by $F(p,q)$ the Schwartz kernel of $F$. Then the kernel $F'(q,p)  =F(p,q)$ of  $F'$ (with the notation convention $F'f(q) =\int  F'(q,p) f(p)\,\d \Vol(p)$) can be written as $F'(q(x,y),p(x))$ that with some abuse of notation we denote again by $F'(y,x)$. Then
\be{sin3a}
F'(y,x) := (2\pi)^{-1}\int e^{\i y^n \eta_n}  \tilde  F'(y',\eta_n,x)\, \d \eta_n,
\ee
where $\tilde F'$ is the partial Fourier transform of $F$ w.r.t.\ $y^n$, and there is no summation in $y^n \eta_n$.  By Theorem~\ref{thm_kernel} and \r{FTz},
\be{sin3a1}
\tilde F'(y',\eta_n,x) =  \pi^{1/2}e^{-\i \pi/4}\big( h(\eta_n)+\i h(-\eta_n) \big)|\eta_n|^{-1/2} G(x,y',\eta_n)
\ee
where $G$ is a symbol w.r.t.\ $\eta_n$, smoothly depending on $(x,y')$ with principal part
\[
G_0:=  W_\Sigma \frac{\sqrt2}{\sqrt{AD}}.
\]
Moreover, by Theorem~\ref{thm_kernel}, $G$ has an expansion it terms of positive powers of $|\eta_n|^{-1/2}$. In particular, $G-G_0$ is an amplitude of order $-1/2$ that contributes a conormal distribution in the class $I^{-n/2-1/2}(M\times M, \Sigma;\mathbf{C})$, see, e.g., \cite[Thm~18.2.8]{Hormander3}. 
By the calculus of conormal singularities, see e.g.,  \cite[Theorem~18.2.12]{Hormander3}, the kernel of $F\M$ is of conormal type at $y^n=0$  as well, with a principal symbol given by that of $F$ multiplied by $\m|_{y^n=0, \eta'=0}$. That principal symbol coincides with the full one modulo conormal kernels of order $1$ less that the former, see the expansions in \cite{Hormander3} preceding Theorem~18.2.12. 
Since we assumed that $r_0$ is an even homogeneous function of $\eta$ of order $0$, 
$r_0(y',0,0,\eta_n)$ is a function of $y'$ only for $\eta$ in a conic neighborhood of $(0,\pm 1)$, equal to $r(y,0,0,1)$.  Therefore,  the principal part of $\m(y,D_y)F'(\cdot,x)$ is
\be{sin4}
(2\pi)^{-1} \int e^{\i y^n \eta_n}   G_0(x,\eta')  r_0(y',0,0,1)  \, \d \eta_n =  
 W_\Sigma \frac{\sqrt2}{\sqrt{AD}} r_0(y',0,0,1) \delta(y^n),
\ee
and the latter is in $I^{-n/2+1/2}(M\times M, \Sigma;\mathbf{C})$. The ``error'' is determined by the next term of the principal symbol of the composition $FR$ with $G$ replaced by $G_0$, that is of order $1$ lower and by the contribution of $G=G_0$ that is of order $-1/2$ lower. Since the coordinates $(y',y^n)$ depend on $p$, as well, $r_0(y',0,0,1)$ is actually the restriction of $r_0$ to $\mathcal{N}^*\Sigma(p)$. 
 So we  proved the following.

\begin{lemma}\label{lemma_FR}
Let $r_0$ be as in \r{sin_q}. Then modulo $I^{-n/2}(M\times M, \Sigma;\mathbf{C})$, $F\M \in I^{1/2-n/2}(M\times M, \Sigma;\mathbf{C})$ reduces to the  Radon transform
\[
F\M f(p) \simeq \int_{\Sigma(p)}af\, dS,\quad a :=   
 \left.  r_0\right|_{\mathcal{N}^*\Sigma(p)} W_\Sigma \frac{\sqrt2}{\sqrt{AD}}  ,
\]
where $\d S$ is the Riemannian surface measure on $\Sigma(p)$ that we previously denoted by $\d\Vol_{\Sigma(p)}$. 
\end{lemma}

In two dimensions, this is an X-ray type of transform. In higher dimensions, this is a Radon type of transform on the family of codimension one surfaces $\Sigma(p)$.

In what follows, $n=2$. 

We will compute $RF^* F R$ next. We have
\be{sin5}
\int FR f \overline{FR h}\,\d\Vol  \simeq\int_M \int_{\Sigma(p)}(af)(z')\, dS(z')\int_{\Sigma(p)}(\bar a \bar h)(q)\, dS(q)\, \d\Vol(p)
\ee
modulo terms of the kind $(Pf,h)$, where $P$ is a \PDO\ of order $-3/2$ or less.  

In the latter integral, $p$ parameterizes the curve $\Sigma(p)$, while $q\in \Sigma(p)$ parameterizes a point on it. Another parameterization is by $p$ and $\xi\in S_p^*M$ with $\xi$ oriented positively; then $q=\exp_p(v)$, where $v\in \Sigma(p)$ and $\xi(v)=0$. For the Jacobian of that change we have 
\be{sin6}
\d S(q)\, \d \Vol(p) =  D\,\d \Vol_{S(p)}(v)\, \d \Vol(p) = \frac{|v|D}{\cos\phi}\,\d \sigma_p(\xi)\,\d \Vol(p),
\ee
and we recall that $\d\sigma_p$ denotes the surface measure on $S_pM$, that in this case is a circle. 
The canonical map $(p,\xi)\to (q,\eta)$ is symplectic, and therefore preserves the volume form $\d p \, \d \xi$. Set 
\be{sin7}
K := |\eta(p,\xi)|/|\xi|. 
\ee
Then this map takes $S^*M$ into   $\{(q,\eta)\in T^*M; \; |\eta|=K\}$. Project that bindle to the unit circle one, and set $\hat\eta = \eta/|\eta|$. Then we have the map $(p,\xi)\to (q,\hat \eta)$, and $\d \Vol(p)\, \d \sigma_p(\xi) = K^2\d \Vol(q) \, \d\sigma_q ({\hat\eta})$.

When we perform those changes of variables in \r{sin5}, we will have
\be{sin8}
\d S(q) \, \d \Vol(p) = \frac{|w|DK^2}{\cos\phi} \d \Vol(q)\, \d \sigma_q({\eta}),
\ee
where $p\in M$, $q\in \Sigma(p)$,  $(q,\eta)\in S^*M$, and we removed the hat over $\eta$.  Let $w$ is the corresponding vector in $S(q)$ normal to $\eta$. That parameterizes the curves $\Sigma(p)$ over which we integrate by initial points $q$ and unit conormal vectors $\eta$. The latter can be replaced by unit tangent vectors $\hat w = w/|w|$; then $\d \Vol(q)\, \d \sigma_q({\eta}) = \d\Vol(q)\, \d \sigma_q({\hat w})$. Let us denote the so parameterized curves
by $c_{q,\hat w}(s)$, where $s$ is an arc-length parameter.

It remains to notice that the integral w.r.t.\ $z'\in \Sigma(p)$ is an integral w.r.t.\ the arc-length measure on $\Sigma(p)$, that we denote by $s$. Then performing the change of the variables $(p,q,z')\mapsto (q,\hat w, z')$ in \r{sin5}, we get 
\be{sin8a}
\int FR f \overline{FR h}\,\d\Vol 
  \simeq \int_{\R\times S_qM\times M   }  (af)(c_{q,\hat w}(s)) 
\bar a(q,-\hat w) \bar h(q) 
 \, \d s\, 
\frac{|w|DK^2}{\cos\phi} \, \d \sigma_q(\hat w) \,\d \Vol(q).
\ee
Therefore, we get as in \eqref{Xstar2}, \r{7.1},
\be{sin9}
\begin{split}
R^*F^* F R f(q) & \simeq \frac{1}{\sqrt{\det(g(q))}}\int    a\bar a \frac{|w|DK^2}{\cos\phi} \frac{f(q')}{\rho(q,q')} \,\d\Vol(q')\\
&  \simeq \frac{1}{\sqrt{\det(g(q))}}\int  \big|r_0|_{\mathcal{N}^*\Sigma(p)}\big|^2    |W_\Sigma|^2 \frac{2|w|K^2}{A\cos\phi} \frac{f(q')}{\rho(q,q')}  \,\d\Vol(q').
\end{split}
\ee
For the directional derivatives of $\det\d \exp_p(v)= -J'/|v|$, see \r{DE}, we have that the derivative along the radial ray is $|J'(1)|/|v|$ by absolute value, while the derivative in the direction of $S(p)$ vanishes. That implies
\[
A\cos\phi = |J'(1)|/|w|= K/|w|.
\]
Therefore,
\be{sin10}
R^*F^* F R  f(q)   \simeq \frac{1}{\sqrt{\det(g(q))}}\int 2 K  \big|r_0|_{\mathcal{N}^*\Sigma(p)}\big|^2    |W_\Sigma|^2 |w|^2
 \frac{f(q')}{\rho(q,q')}  \,\d\Vol(q').
\ee
Here $(p,v)$ is defined as follows. It is the point in $SM$ that lies on the continuation of the geodesic through $q$, $q'$ to its conjugate point near $p_0$, The weight $\kappa$ restricts $q'$ to a small neighborhood of $\gamma_0$. Next, $A_2$ restricts $q'$ near $q_0$. 

We compare \r{sin10} with \r{7.1} and \r{a2}. Notice that the Jacobian term in \r{7.1} at the diagonal equals $\sqrt{\det g}$ and therefore cancels the factor in front of the integral in the calculation of the principal symbol. 
We therefore proved the following. 

\begin{lemma}\label{pr_F\M } Let $n=2$. Then 
$R^*F^* FR$ is a \PDO\ of order $-1$ with principal symbol modulo $S^{-3/2}$ at $(q,\eta)$ near $(q_0,\eta_0)$ given by
\[
4\pi K|\eta|^{-1}\big|r_0|_{\mathcal{N}^*\Sigma(p)}\big|^2 |\kappa(p,v/|v|)|^2 |\kappa(q,-w/|w|)|^2
\]
Here $w/|w|$ is a continuous choice of a unit vector normal to $\eta$ at $q$, so that $(q,w/|w|)=(q_0,w_0/|w_0|)$ when $(q,\eta)=(q_0,\eta_0)$, and $v/|v|$ is a parallel transport of $-w/|w|$ from $q$ to its conjugate point $p$ along the geodesic $\gamma_{q,w}$.
\end{lemma}

Later we use the notation $w=\eta^\perp/|\eta^\perp|$, and $v=\xi^\perp/|\xi^\perp|$. 

\begin{proposition}\label{pr_2D_symbol} Let $n=2$. Then 
\[
\Id - A_2^{-1}F^*A_1^{-1}F
\]
is a \PDO\ of order $-1/2$.
\end{proposition}

\begin{proof} 
We apply  Lemma~\ref{pr_FR} with $\pi^{-1/2}e^{\i \pi/4}|\eta|^{1/2}r_0$ being the principal symbol of $A_2^{-1/2}$, see \r{sin_q}, where $A_2^{-1/2}$ is a parametrix of $A_2^{1/2}$ near  $(q_0,\pm \eta_0)$. To this end, choose 
\[
\pi^{-1/2}e^{\i \pi/4}(2\pi)^{-1/2} r_0(q,\eta) =(2\pi)^{-1/2}|\kappa(q,\eta^\perp/| \eta^\perp|)|^{-1},
\]
see \r{pr2D}. Note that $\kappa(q,w/|w|)=k(p,-v/|v|)=0$ because of the assumption on $\supp\kappa$. Then $\big|r_0|_{\mathcal{N}^*\Sigma(p)}\big| = 2^{-1/2}|\kappa(q,-w/|w|)|^{-1}$, where $w$ is as in \r{qw}.  The choice of $r_0$ yields $RR^*=A_2^{-1/2}$ mod $\Psi^{-1}$. So  Lemma~\ref{pr_FR} implies  that $R^*F^*FR$, and therefore  $RR^*F^*F$ and $A_2^{-1}F^*F$, have principal symbol 
\[
\sigma_{\rm p}(A_2^{-1}F^*F)(q,\eta) =  {2}\pi  K |\kappa(p,\xi^\perp/|\xi^\perp|)|^2/|\eta|
\]
We only need to insert $A_1^{-1}$ between $F^*$ and $F$. By  \cite[Thm~25.3.5]{Hormander4},  modulo \PDO s of order $1$ lower, the principal symbol of $A_2^{-1}F^*A_1^{-1}F$ is given by that of $A_2^{-1}F^*F$ multiplied by the principal symbol $\left(2\pi |\kappa(p,v)|^2|/|\xi|\right)^{-1}$ of $A_1^{-1}$ pushed forward by the canonical map of $F$. In other words,
\[
\sigma_p(A_2^{-1}F^*A_1^{-1}F)(q,\eta) = 
\frac{2\pi|\kappa(p,\xi^\perp/|\xi^\perp|)|^2}{|\eta|}K 
 \left[ 2\pi|\kappa((p,\xi^\perp/|\xi^\perp|)|^{2}/|\xi(q,\eta)|)\right]^{-1} = 1 .
\]
\end{proof}

The following lemma is needed below for the proof of Theorem~\ref{thm_cancel}. 

\begin{lemma}\label{lemma_XX}
Let $\kappa_1$ and  $\kappa$  both satisfy the assumptions for $\kappa$ in the Introduction, and let $\kappa(p_0,\theta_0)\not=0$. 
Let $\chi\in \Psi^0$ have essential support near $(p_0,\pm \xi_0)\cup (q_0,\pm \eta_0)$ and Schwartz kernel in $(U_1\times U_1)\cup (U_2\times U_2)$. 
Then there exists  a zero order classical \PDO\ $Q$ with the same support properties so that  
\[
Q X^*_{\kappa}X_\kappa\chi = X^*_{\kappa_1}X_\kappa\chi , \quad \text{mod}\; I^{-3/2}(M\times M, \;\Delta\cup \mathcal{N}^*\Sigma, \; \mathbf{C}),
\]
where $\Delta$ is the diagonal. In particular, $Q X^*_{\kappa}X_\kappa\chi - X^*_{\kappa_1}X_\kappa\chi :H^s\to H^{s+3/2}$ is bounded for any $s$. 
\end{lemma}

\begin{proof} We define $Q=Q_1+Q_2$ where $Q_{1,2}$ have Schwartz kernels in $U_1\times U_1$ and $U_2\times U_2$, respectively. Following the notation convention in \r{Fmatrix}, $Q=\text{diag}(Q_1, Q_2)$.

Then we choose $Q_1$ to have principal symbol
\be{k11}
\bar \kappa_1(p,\xi^\perp/|\xi^\perp|)/ \bar\kappa(p,\xi^\perp|/|\xi^\perp)
\ee
in a conic neighborhood of $(p_0,\pm\xi_0)$ with the same choice of $\xi^\perp$ as in \r{pr2D}. Next, we choose $Q_2$ with a principal symbol
\be{k12}
\bar \kappa_1(q,\eta^\perp/|\eta^\perp|)/\bar \kappa(q,\eta^\perp|/|\eta^\perp)
\ee
in a conic neighborhood of $(q_0,\pm\eta_0)$. Then 
\[
Q X^*_{\kappa}X_\kappa= \begin{pmatrix}
 Q_1A_1&Q_1F\\Q_2F^*&Q_2A_2
\end{pmatrix}.
\]
Then, see \r{pr2D}, 
\[
\sigma_{\rm p}(Q_1A_1) = 2\pi( \bar\kappa_1\kappa)(p,\xi^\perp/|\xi^\perp|)|
, \quad 
\sigma_{\rm p}(Q_2A_2) = 2\pi (\bar\kappa_1\kappa)(q,\eta^\perp/|\eta^\perp|).
\]
For $Q_1F$, $Q_2F^*$, we use the arguments used in the proof of Lemma~\ref{lemma_FR}. A representation of the Schwartz kernel of $F'$ as a conormal distribution is given by \r{sin3a}. The composition $Q_2F^*$ then is of the same conormal type with a principal symbol equal to the complex conjugate of that of $F'$ multiplied by the symbol \r{k12} restricted to $\mathcal{N}^*\Sigma$. This replaces $k^\sharp= \bar\kappa$ in \r{WS} by $\bar \kappa_1$. Since in \r{WS}, $\kappa^\sharp=\bar\kappa$ we get that $Q_2F^*$ is of the same conormal type with leading singularity as in Theorem~\ref{thm_kernel}, with 
\[
W_\Sigma = |v|^{-1} \bar \kappa(p,v/|v|)\kappa_1(q,-w/|w|).
\]
This is however the leading singularity of $\chi_2X_{\kappa_1}^*X_\kappa\chi_1$.

The proof for $Q_1F$ is the same with the roles of $p$ and $q$ replaced. 
\end{proof}

\bigskip

\subsection{Cancellation of singularities on Riemannian surfaces} Assume in all dimensions that  there are no conjugate points on the geodesics in $M$, and that $\bo$ is strictly convex. Let $M_1\supset M$ be an extension of $M$ so that the interior of $M_1$ contains $M$ be as in Remark~\ref{remark_global}. Then if $\kappa\not=0$, 
\be{L1}
\|f\|_{L^2(M)}\le C\|X^*Xf\|_{H^1(M_1)}+ C_k\|f\|_{H^{-k}(M)}, \quad \forall f\in L^2(M),
\ee
for all $k\ge0$,  
see \cite{SU-Duke, FSU}, and \cite{SU-AJM} for a class of manifolds with conjugate points. When we know that $X$ is injective, for example when the weight is constant; then we can remove the $H^{-k}$ term. 
The same arguments there show that for any $s\ge0$,
\be{L2}
\|f\|_{H^s(M)}\le C\|X^*Xf\|_{H^{s+1}(M_1)}+ C_k\|f\|_{H^{-k}(M)},\quad \forall f\in H_0^s(M). 
\ee
Consider $Xf$ parameterized by points in $\partial_+SM_1$, that defines Sobolev spaces for $Xf$ as in section~\ref{sec_map}. Then 
\be{L3}
\|f\|_{H^s(M)}\le C\|Xf\|_{H^{s+1/2}(\partial_+SM_1)}+ C_k\|f\|_{H^{-k}(M)},\quad \forall f\in H_0^s(M), \; s\ge0.
\ee
Indeed, in Proposition~\ref{pr_map}, one can complete $M_1$ and $\mathcal{H}$ to closed manifolds, and then we would get that $X^*:H^{s}\to H^{s+1/2}$ is bounded. Then \r{L3} follows by \r{L2}. 
Estimate \r{L3} is  sharp in view of Proposition~\ref{pr_map}. In the following theorem, we show that \r{L1}, \r{L3} fail in the 2D case, with a loss at least of one derivative in the first one, and $1/2$ derivative in the second one. 

\begin{theorem}  \label{thm_cancel}
Let $n=2$, and let $\gamma_0$ be a geodesic of $g$ with conjugate points satisfying the assumptions in section~\ref{sec_2}. Then for each $f_2\in H^s(M)$, $s\ge0$,  with $\WF(f_2)$ in  a small neighborhood of $(q_0,\pm\eta^0)$, there exists $f_1\in H^s(M)$  with $\WF(f_1)$ in  a some neighborhood of $(p_0,\pm \xi^0)$ so that 
\[
Xf\in H^{s+3/4}\quad \mbox{and} \quad   X^*Xf\in H^{s+3/2}, \quad\text{where $f:= f_1+f_2$}. 
\]
In particular, if $(M,g)$ is a non-trapping Riemannian surface with boundary  with fold type of conjugate points on some geodesics, none of the inequalities \r{L1}, \r{L3} can hold. 
\end{theorem}
\begin{remark}
It is an open problem whether  we can replace $H^{s+3/2}$ and $H^{s+2}$ above with $C^\infty$. See Section~\ref{sec_circle} for an example where this can be done. 
\end{remark}

\begin{remark}
If there are no conjugate points, one has $Xf\in H^{s+1/2}$, $X^*Xf \in H^{s+1}$. Therefore, the conjugate points are responsible for an $1/4$ derivative smoothing for $X f$, and an $1/2$ derivative  smoothing for $X^*Xf$ 
\end{remark}

\begin{proof}
Let $f_2$ be as in the theorem. Set
\[
f_1 = -A_1^{-1}Ff_2,
\]
where, as before, $A_1^{-1}$, $A_2^{-1}$ are parametrices of $A_{1,2}$ in  conic neighborhoods of  $(p_0,\pm \xi_0)$ and $(q_0,\pm\eta_0)$, respectively. Then $f_1$ belongs to $H^s$ and has a wave front set in small neighborhood of $(p_0\pm\xi_0)$, by Theorem~\ref{thm_main}. By construction and by \r{N1},
\be{L4}
\chi_1 X^*Xf\in C^\infty.
\ee
Next, by \r{L4},
\[
A_2 f_2 +F^*f_1 = A_2 f_2 -F^* A_1^{-1}Ff_2 =(A_2  -F^* A_1^{-1}F)f_2.
\]
The operator in the parentheses is a \PDO\ of order $-3/2$ by Proposition~\ref{pr_2D_symbol}. Therefore, see \r{N2},
\[
\chi_2 X^*Xf = A_2 f_2 +F^*f_1 \in H^{s+3/2}. 
\]
We therefore get  $X^*Xf\in H^{s+3/2}(U_1\cup U_2)$. 

To prove $Xf\in H^{s+1}$, note first that above we actually proved that
\be{act}
X^*X(\Id - A_1^{-1}F)\chi : H^s(U_2) \longrightarrow H^{s+3/2}(U_1\cup U_2)
\ee
is bounded, being a \PDO\ of order $-3/2$, where $\chi$ denotes a zero order \PDO\ with essential support in a small neighborhood of $(p_0,\pm\eta_0)$ and Schwartz kernel supported in $U_2\times U_2$. 

Our goal is to show that
\[
X(\Id - A_1^{-1}F)\chi : H^s(U_2) \longrightarrow H_0^{s+3/4}(\mathcal{H})
\]
is bounded. It is enough to prove that 
\be{act2}
\chi^*(\Id - A_1^{-1}F)^*X^* P_{2s+3/2} X(\Id - A_1^{-1}F)\chi : H^s(U_2) \longrightarrow H^{-s}(U_2)
\ee
for any \PDO\ $P_{2s+3/2}$ of order $2s+3/2$ on $\mathcal{H}$. All adjoints here are in the corresponding $L^2$ spaces. By \r{act},
\be{act3}
Q_{2s+3/2}X^*X (\Id - A_1^{-1}F)\chi : H^s(U_2) \longrightarrow H^{-s}(U_2)
\ee
is bounded for any \PDO\ $Q_{2s+3/2}$ of order $2s+3/2$.

To deduce \r{act2} from \r{act3}, it is enough to ``commute'' $X^*$ with $P_{2s+3/2}$ in \r{act2}. Let $2s+3/2$ be a non-negative integer first. As in the proof of Proposition~\ref{pr_map}, we use the fact that $X^* P_{2s+3/2}= (P_{2s+3/2}^*X)^*$, and $P_{2s+2}^*Xf$ is a finite  sum of X-ray transforms with various weights of derivatives of $f$ of order not exceeding $2s+2$. Thus we can write
\be{act4}
X^* P_{2s+2} =  \sum \tilde Q_j X_j^*,
\ee
where $Q_j$ are differential operators on $\mathcal{H}$ of degree $2s+3/2$ or less, and $X_j$ are like $X$ in \r{1.1} but with different weights still supported where $\kappa$ is supported. By Lemma~\ref{lemma_XX}, $\tilde Q_jX_j^*X=  R_jX^*X$, where $R_j$ is a \PDO\ of the same order as $\tilde Q_j$.  The proof of \r{act2} is then completed by the observation that $\chi^*(\Id - A_1^{-1}F)^*$ maps continuously $H^s$ into itself, since the canonical relation of $F$ is canonical graph. 
\end{proof}

\section{Examples}\label{sec_ex}
In this section, we present a few examples. We start in Section~\ref{sec_circle} with the fixed radius circular transform in the plane, where we can have cancellation of singularities similarly to Theorem~\ref{thm_cancel} but we show that this happens to any order. Then we consider in Section~\ref{sec_sphere} the geodesic X-ray transform on the sphere, where the conjugacy is not of fold type, but a similar result holds. Next, in Section~\ref{sec_mag}, we study an example of magnetic geodesics in the Euclidean space $\R^3$ with a constant magnetic field. We show that then the canonical relation of $F$ a canonical graph, and therefore, one can resolve the singularities. Finally, in Section~\ref{sec_prod}, we present an example of a Riemannian manifold of product type where the graph condition is violated.

\subsection{The fixed radius circular transform in the plane} \label{sec_circle} 
Let $R$ be the integral transform in $\R^2$ of integrating functions over circles of radius $1$. We fix the negative  orientation on those circles; then for each $(x,\xi)\in S\R^2$, there is a unique unit circle passing through $x$ in the direction of $\theta$. It is very easy to see that the first conjugate point appears at ``time'' $\pi$. The next one is at $2\pi$, that equals the period of the curve. If one originally chooses $f$ supported near, say $(0,0)$ and $(2,0)$; and chooses $\gamma_0$ to be the arc of the circle that is a small extension of $\{|x_1-1|^2+x_2^2=1, \; x_2\ge0\}$,  then we are in the situation studied above. On the other hand, if we do not impose any assumptions on $\supp f$, we will get contributions that are smoothing operators only. Therefore, we do not need to restrict $\supp f$.

The conjugate locus  $\Sigma(x)$ is the circle 
\[
\Sigma(x) = \{y; \; |y-x|=2\}
\]
that is the envelope of all circles of radius $1$ passing through $x$. It follows immediately that
\[
S_x(v)  = \{v;\; |v|=\pi\}, \quad N_x(v) = \R e^{\i\alpha}(2/\pi ,-1),
\]
where we used complex identification to denote rotation by the angle $\alpha = \arg(v)$. Hence, $S$ is a fold conjugate locus. The other assumptions of the dynamical system are easy to check.

It is much more natural to parametrize those circles by their centers, we use the notation $C(x)$. Then
\be{1}
Rf(x) = \int_{C(x)} f\, \d\ell = \int_{|\omega|=1} f(x+\omega)\, \d\ell_{\omega} = \int_0^{2\pi} f(z+e^{\i\alpha})\,\d\alpha.
\ee
Those circles are also magnetic geodesics w.r.t.\ the Euclidean metric and a constant non-zero Lorent\-zian force. 
 Note that the ``geodesics'' are naturally parametrized by a point in $\R^2$ as well (and that parametrization reflects the choice of the measure w.r.t.\ which we take $R^*$, it is not hard to see that this is the same measure that we had before).

\subsubsection{$R$ as a convolution} It is well known and easy to see that $R$ is a convolution with the delta function $\delta_{S^1}$ of the unit circle
\[
Rf = \delta_{S^1}*f.
\]
Fourier transforming, we get
\be{2}
R = 2\pi\mathcal{F}^{-1} J_0(|\xi|) \mathcal{F},
\ee
where $J_0$ is the Bessel function of order $0$. This shows that
\be{3}
R^*R = (2\pi)^2\mathcal{F}^{-1} J_0^2(|\xi|) \mathcal{F}.
\ee
Note that $J_0^2(|\xi|)$ is not a symbol because it oscillates. In principle, one can use this representation to analyze $R^*R$.

\subsubsection{Integral representation} We write
\be{4}
\begin{split}
(Rf,Rh) &= \int\int_{|\omega|=1} f(x+\omega)\,\d\ell_\omega \int_{|\theta|=1} \bar h(x+\theta)\,\d\ell_\theta\,dx\\
 &= \int \int_{|\omega|=1} \int_{|\theta|=1}f(y+\omega-\theta)\bar h(y)\,  \d\ell_\omega\, \d\ell_\theta\, \d  y.
\end{split}
\ee
Therefore,
\be{5}
R^*Rf(x) = \int_{|\omega|=1} \int_{|\theta|=1}f(x+\omega+\theta) \,  \d\ell_\omega\, \d\ell_\theta,
\ee
compare with  \r{Xstar}.

We will make the change of variables $z=\omega+\theta$. For $0<|z|<2$, there are exactly two ways $z$ can be represented this way. 
Write $\omega=e^{\i\alpha}$, $\theta=e^{\i\beta}$. Since $\d\ell_\omega= \d\alpha$, $\d\ell_\theta= \d\beta$, and $\d z_1 \wedge\d z_2 = (-2 i)^{-1}\d z\wedge\d \bar z$, we get
\[
\begin{split}
\d z_1 \wedge\d z_2 &= \frac{1}{-2\i}\left(\i  e^{\i\alpha}\d\alpha +   \i  e^{\i\beta}\d\beta        \right)\wedge \left(-\i  e^{-\i\alpha}\d\alpha -   \i  e^{-\i\beta}\d\beta        \right) = \sin(\beta-\alpha) \,\d\alpha\wedge \d\beta\\
  &= \sin(\beta-\alpha)\, \d\ell_\omega\wedge \d\ell_\theta.
\end{split}
\]
It is easy to see that $|\beta-\alpha|$ equals twice the angle between $z=\omega+\theta$ and $\theta$. Let $r=|z|$. Then $r/2= \cos\frac{|\alpha-\beta|}2$. Elementary calculations then lead to
\[
\sin|\alpha-\beta|= \frac{r}2\sqrt{4-r^2}.
\]
Therefore, \r{5} yields the following.
 
\begin{proposition}  \label{pr1} Let $R$ be  the circular transform defined above. Then 
\be{6}
R^*Rf(x) = \int_{r<2} \frac{4}{r\sqrt{4-r^2}}f(y)\,\d y, \quad r := |x-y|.
\ee
\end{proposition}

\subsubsection{$R^*R$ as an FIO} 
The kernel has  singularities near the diagonal $x=y$, and also near 
\[
\Sigma = \{|x-y|=2\}.
\] 
That singularity is of the type $(2-|x-y|)^{-1/2}$, and for a fixed $x$  the expression $2-|x-y|$ measures the distance from the circle $\Sigma(x)$ to the point $y$ inside that circle. We therefore get the same singularity as in Theorem~\ref{thm_kernel}. Note also that
\be{Ap1}
\mathcal{N}^*\Sigma = \{ (x,x\pm2\xi/|\xi|,\xi,-\xi); \;  \xi\in \R^2 \setminus 0 \}.
\ee

Based on Proposition~\ref{pr1}, and Theorem~\ref{thm_main}, we conclude that $R^*R$ is an FIO of order $-1$ with a canonical relation $\mathcal{C}$ of the following type. We have that $(x,\xi,y,\eta)\in \mathcal{C}$ if and only if $(y,\eta)=(x,\xi)$ (that gives us the \PDO\ part), or $(y,\eta) = (x\pm2\xi/|\xi|,\xi)$.

This can also be formulated also in the following form. 
\begin{theorem}\label{thm1} Let $R$ be  the circular transform defined above. Then, modulo $\Psi^{-\infty}$, 
\be{eq:thm9}
R^*R = A_0 + F_+ + F_-,
\ee
where $A_0$, $F+$ and $F_-$ are Fourier multipliers with the properties

(a) $A_0=4\pi |D|^{-1}$ mod $\Psi^{-1}$; 

(b) $F_\pm$ are elliptic FIOs of order $-1$ with canonical relations of a graph type given by 
\be{9a}
\mathcal{F}_\pm: 
(x,\xi) \mapsto (x\pm 2\xi/|\xi|,\xi).
\ee

(c) $F_-=F_+^*$. 
\end{theorem}

\begin{proof}
We start with the Fourier multiplier representation \r{2}. The leading term of $(2\pi)^2 J_0^2(|\xi|)$ is 
\be{exp}
\frac{8\pi}{|\xi|}\cos^2(|\xi-\pi/4) = \frac{8\pi}{|\xi|}(1+\sin(2|\xi|) 
= 2\pi\left( \frac2{|\xi|} +      \frac{e^{2\i|\xi|}  }{\i |\xi|} -  \frac{e^{-2\i|\xi|}  }{\i |\xi|} \right).
\ee
Those three terms are the principal parts of the operators in \r{eq:thm9}. The first one gives $4\pi |D|^{-1}$, while the second and the third one are FIOs with phase functions $\phi_\pm= (x-y)\cdot\xi \pm 2|\xi|$. A direct calculation show that the canonical relations of $F_\pm$ are given by \r{9a}, indeed. For the complete proof of the theorem, we need the  full asymptotic expansion of $J_0$. 

We recall the well known expansion of $J_0(z)$ for $z\to\infty$: 
\[
J_0(z) \sim \sqrt{2/(\pi z)}\left( P(z)\cos(z-\pi/4)-Q(z)\sin(z-\pi/4)   \right),
\]
where
\[
P(z) \sim \sum_{k=0}^\infty p_kz^{-2k}, \quad Q(z) \sim \sum_{k=0}^\infty q_kz^{-2k-1},
\]
with some (explicit) coefficients $p_k$, $q_k$. In particular, $p_1=1$, $q_1=-1/8$. 
Then
\[
\begin{split}
(2\pi)^2 J_0^2(z) &\sim \frac{2\pi}{z} \left( (P+\i Q)e^{\i (z-\pi/4)} + (P-\i Q) e^{-\i (z-\pi/4)} \right)^2\\
         &\sim  \frac{2\pi}{z} \left( -\i  (P+\i Q)^2e^{2\i z} + \i (P-\i Q)^2 e^{-2\i z}  + 2P^2+2Q^2  \right).
\end{split}
\]
We set
\be{A0}
A_0= 4\pi|D|^{-1}\left(P^2(|D|)+Q^2(|D|)\right), \quad 
F_\pm = \mp 2 \pi\i |D|^{-1} \big(P(|D|)\pm\i Q(|D|)   \big)^2e^{\pm 2\i |D| }.
\ee
This completes the proof. 
\end{proof}

We will now connect this to Theorem~\ref{thm_main}. Let $p_0=(0,0)$, $q_0=(2,0)$, $v_0=(0,\pi)$, $w_0= (0,\pi)$. Then $v_0\in S(p_0)$. Choose $\xi_0= (1,0)$, conormal to the conjugate locus $\Sigma(q_0) = \{|x-q_0|=2\}$ at $p_0$; and choose $\eta_0= (1,0)$, conormal to the conjugate locus $\Sigma(p_0) = \{|x-p_0|=2\}$ at $q_0$. The directions of $\xi_0$, $\eta_0$ reflects the choice of the orientation we made earlier. We refer to Figure~\ref{fig:caustics2}.

\begin{figure}[h] 
  \centering
  \includegraphics[bb=0 0 454 401,width=3.49in,height=3.08in,keepaspectratio]{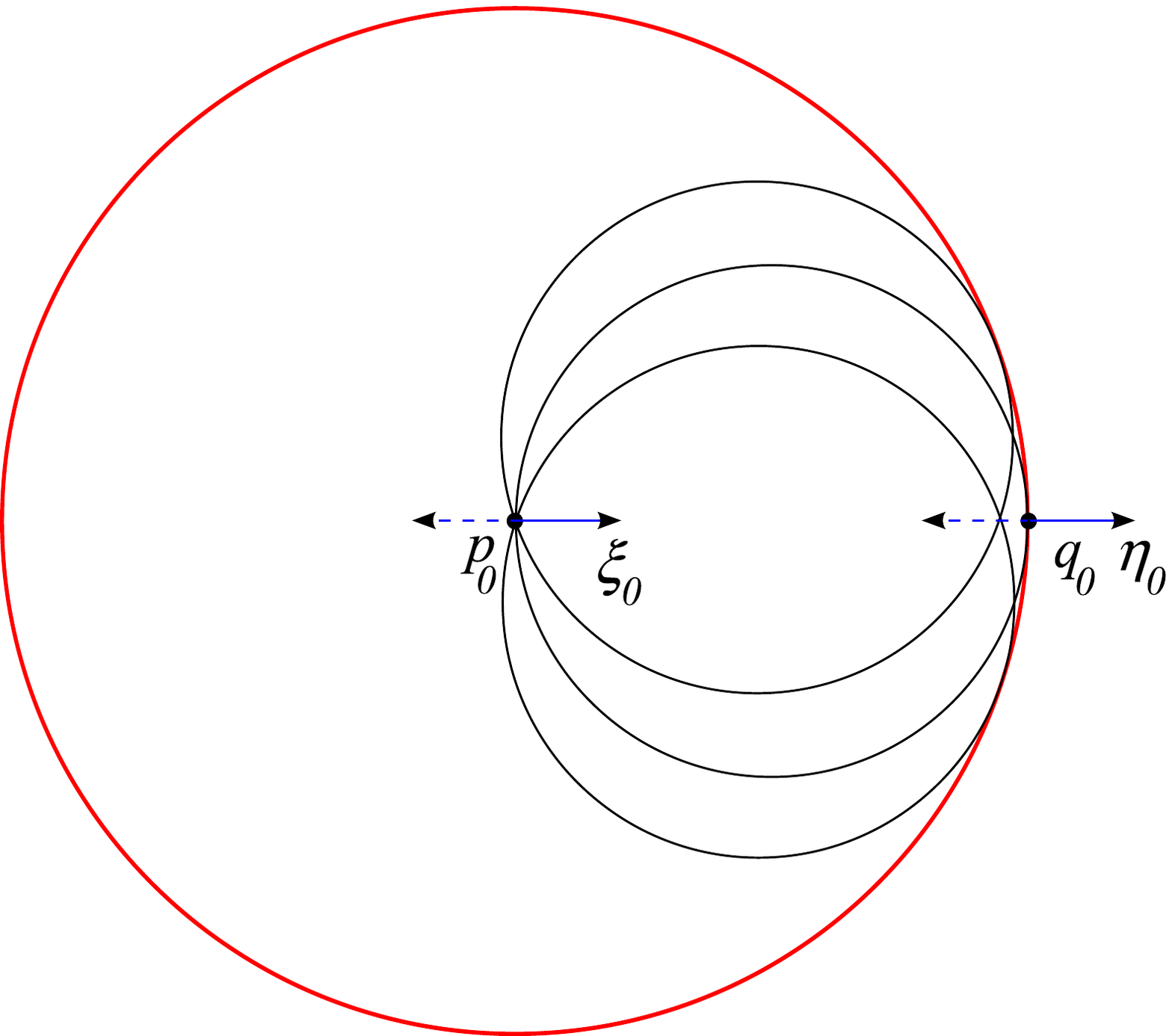}
  \caption{}
  \label{fig:caustics2}
\end{figure}

If we localize $R$ near $v=v_0$, then the pseudo-differential part of $R^*\chi R$ is $(1/2)A_0$, see \r{thm_PDO_eq1}. Therefore, in the notation of Theorem~\ref{thm_main}, 
\[
A=\frac12 A_0, \quad F= F_++F_-. 
\]
The canonical relation of $F_+$ maps $(p_0,\xi_0)$ into $(q_0,\eta_0)$, see  Figure~\ref{fig:caustics2}, while that of $F_-$ maps $(p_0,-\xi_0)$ into $(q_0,-\eta_0)$. This is consistent with the results in Theorem~\ref{thm_main}, where the Lagrangian has two disconnected components located   near $(p_0,q_0, \pm\xi_0,\mp \eta_0)$.

To analyze the operator \r{main_question}, note first that $A_1=A_2=A_0/2$. Let us first analyze this operator applied to distributions with wave front set near $(q_0,\eta_0)$ but not near $(q_0,-\eta_0)$. Then $F$ reduces to $F_+$ only, and we have, modulo $\Psi^{-\infty}$, 
\[
A_2^{-1}F^*A_1^{-1}F = \frac14 A^{-2}F_+^*F_+=\Id,
\]
see \r{A0}. The analysis near $(q_0,-\eta_0)$ is similar. 
Therefore, we have a stronger version of Theorem~\ref{thm_cancel} in this case: singularities can cancel to any order.

\begin{theorem}\label{thm_circles}
Let $f_1$ be any distribution with $\WF(f_1)$ supported in a small conic neighborhood of some $(x_0,\xi^0)\in T^*\R^2\setminus0$. Then there exists a distribution $f_2$ with $\WF(f_2)$ supported in a small conic neighborhood of $(x_0\pm 2\xi^0/|\xi^0|,\xi^0)$, that is an image of $\WF(f_1)$ under the map $\mathcal{F}_\pm$, so that $R(f_1+f_2)\in C^\infty$ for all unit circles in a neighborhood of the unit circle $C(x_0\pm \xi^0)$. 
\end{theorem}

In other words, for a fixed circle $C_0$ of radius $1$,  there is a rich set of distributions $f$, with any order of singularity at $\mathcal{N}^*C_0$, so that those singularities are invisible by $X$ localized near $C_0$, i.e., $Xf\in C^\infty$.  Explicit examples can be constructed by choosing $f_2(x)=\delta(x-q_0)$, then $Ff_2$ near $p_0$ is just given by the Schwartz kernel of $R^*R$, see \r{6}. To obtain $f_1$, we apply $2A_0^{-1}$ to the result. 

We would like to emphasize on the fact that the theorem provides an example of cancellation of singularities for the localized transform only. As we will see below, $Rf\in C^\infty$ (globally) for $f\in \mathcal{E}'$ implies $f\in C^\infty$. On the other hand, without the compact support assumption, on can construct singular distributions in the  kernel of $R$, using the Fourier transform.  

\subsubsection{The wave front set of a distribution in $\Ker R$}
Now, if $Rf=0$ or more generally, if $Rf\in C^\infty$, one easily gets that 
\be{10}
\mbox{ $\forall f\in \Ker R$, $\WF(f)$ is invariant under the action of the group $\{\mathcal{F}_+^m, \; m\in\mathbf{Z}\}$.  }
\ee
Then, if $f$ is compactly supported (or more generally, smooth outside some compact set), we get that $\WF(f)$ must be empty, i.e., $f\in C^\infty(\R^2)$. In other words, even though recovery of $\WF(f)$ is impossible by knowing $Xf$ locally, as we saw above; the condition  $Xf\in C^\infty$ globally, together with the compact support assumption yielded a global recovery of singularities. Here an important role is played by the fact that $X$ is translation invariant, and in particular, our assumptions are valid for any $(p_0,\theta_0)\in TS\R^2$ that cannot be guaranteed in the general case. Also, the dynamics is not time reversible; therefore for each $(x_0,\xi^0)\in T^*M\setminus 0$ there are two different curves through $x_0$ in our family. The latter is true for general magnetic systems with a non-zero magnetic field, see 
\cite{St-magnetic}.

\begin{remark}
One can see that $R$ is invertible on $L^2(M)$ by using Fourier transform, see \r{2}.  The formal inverse is $1/J_0(|\xi|)$, and conjugating  a compactly supported $\chi$ with the Fourier transform, one gets a convolution in the $\xi$ variable that will smoothen out the zeros of $J_0(|\xi|)$, thus producing a Fourier multiplier with asymptotic  $\sim |\xi|^{1/2}$. In $L^p(\R^2)$ with $p>4$ however it is not invertible, and elements of the kernel include functions with Fourier transforms supported on the circles $J_0(|\xi|)=0$, see also \cite{Thang, AKuchment2009}. 
\end{remark}

Finally, we remark that in this case, one can study $R$ directly, instead of $R^*R=R^2$, with the same methods. Our goal however is to connect the analysis of this transform with our general results. 

\subsection{The X-ray transform on the sphere} Consider the geodesic ray transform on the sphere $S^{n}$. \label{sec_sphere}
The conjugate points are not of fold type, instead they are of blow-down type. Let $J$ be the antipodal map. 

Without going into details, 
we will just mention that then \r{dec} still holds with 
\[
CN = |D|^{-1}-|D|^{-1}J,
\]
with some constant $C$, where the canonical relation of $F$ is the graph of the antipodal map, lifted to $T^*S^2$. Then $CN|D|= \Id-J$. 
The canonical graph is an involution, however (its square is identity), so arguments similar to that in the previous example do not apply. That means that singularities may cancel. In fact, it is known that $R$ has an infinitely dimensional kernel --- all odd functions with respect to $J$. This is a case opposite to the one above. 

In this case $\Sigma$ consists of all antipodal pairs $(x,y)$, and has dimension 2 (and codimension 2), unlike the case above (dimension 3 and codimension 1). On the other hand, $\mathcal{N}^*\Sigma$ still has the same dimension (that is 2n=4, and this is always the case as long as $\Sigma$ is smooth submanifold). One can see that the Lagrangian in this case is still $\mathcal{N}^*\Sigma$.

\subsection{Magnetic geodesics in $\R^3$} \label{sec_mag}
 Consider the magnetic geodesic system in the Euclidean space $\R^3$ with a constant magnetic potential $(0,0,\alpha)$, $\alpha>0$. The geodesic equation is then given by
\be{m1}
\ddot \gamma = \dot\gamma  \times (0,0,\alpha),
\ee
where $\times$ denotes the vector product in $\R^3$. The r.h.s.\ above is the Lorentz force that is always normal to the trajectory and in particular does not affect the speed. We restrict the trajectories on the energy level $1$ that is preserved under the flow. Then we get
\[
\ddot \gamma^1 = \alpha \dot\gamma^2, \quad \ddot\gamma^2= -\alpha\dot\gamma^1, \quad \ddot\gamma^3=0. 
\]
The magnetic geodesics are then given by
\[
\gamma(t) = \gamma(0)+\left(\frac{r}\alpha(\sin(\alpha t+\theta)-\sin\theta) , \frac{r}\alpha(-\cos(\alpha t+\theta)+\cos\theta)  , tz      \right),
\]
where $(r,\theta,z)$ are the cylindrical coordinates of $\dot\gamma(0)$. The unit speed requirement means that
\[
r^2+z^2=1.
\]
The geodesics are then spirals; when $z=0$ then they reduce to closed circles, and when $r=0$ they are vertical lines. 

The parameterization by cylindrical coordinates is singular when $r=0$. Away from that we can use $\theta$, $z$ to parametrize unit speeds. Then in $\exp_p(v)$, we use the coordinates $(t,\theta,z)$ to parametrize $v$, i.e., 
\[
v=t\Big(\sqrt{1-z^2}(\cos\theta,\sin\theta),z\Big).
\]
At $t=0$ we may have additional singularity but this is irrelevant for our analysis since we know that the exponential map has an injective differential near $v=0$. 
An easy computation yields that the conjugate locus is given by the condition $\alpha t=\pi$, i.e.,
\[
S_p(v) = \left\{v;\; |v|=\pi/\alpha\right\},
\]
and this is true for any $p\in \R^3$. This is a sphere in $T\R^n$. 
For $\Sigma_p$ we then get
\be{qg}
\gamma(\pi/\alpha) = p + \alpha^{-1}(-2r\sin\theta,2r\cos\theta,\pi z)
\ee
with $p=\gamma(0)$. This shows that $\Sigma(p)$ is an ellipsoid 
\[
\Sigma = \left\{(p,q); \;  \frac{1}{4}(q_1-p_1)^2 + \frac{1}{4}(q_2-p_2)^2+ \frac{1}{\pi^2}(q_3-p_3)^2 = \alpha^{-2}\right\}.
\]
Then 
\be{NS1}
\mathcal{N}^*\Sigma = \left\{(p,q,\xi,\eta); \;  (p,q)\in \Sigma; \; \xi =c\Big(p_1-q_1,p_2-q_2, \frac4{\pi^2}(p_3-q_3)\Big), \; \eta = -\xi, \; 0\not=c\in\R\right\}. 
\ee
Therefore, given $p$, $\xi$, we can immediately get $q$ as a smooth function of $(p,\xi)$, and we can obtain $v$ so that $\exp_p(v)=q$ by \r{qg}, where the l.h.s.\ is $q$. Therefore, $(p,\xi)\mapsto v$ is a smooth map, and therefore $(p,\xi)\mapsto (q,\eta)$ is a smooth map, too. The later also directly follows from \r{NS1}, since $\eta=-\xi$. 

We therefore get that $F$ is an FIO of order $-3/2$ with a canonical relation
\be{q3}
(p,\xi) \mapsto (q,\xi), 
\ee
where $q$ can be determined as described above. A geometric description of $q$ is the following: $q$ is one of the two points on the ellipsoid $\Sigma$, where the normal is given by $\xi$.  The choice of one out of the two points is determined by the choice of the initial velocity $v_0$ near which we localize; changing $v_0$ to $-v_0$ would alter that choice. Since \r{q3} is a diffeomorphism, $F$ is of canonical graph type, and therefore maps $H^s$ to $H^{s+3/2}$. In contrast, $A_{1,2}$ are elliptic of order $-1$, thus they dominate over $F$. By Corollary~\ref{cor_sing}, $X$ can be inverted microlocally in the setup described in Section~\ref{sec_2}.

\subsection{Fold caustics on product manifolds} \label{sec_prod}
Let 
$(M,g) = (M',g')\times (M'',g'')$ be a product of two Riemannian manifolds. The geodesics on $M$ then have the form
\[
\gamma_{p,v}(t) = (\gamma'_{p',v'}(t),\gamma''_{p'',v''}(t)).
\]
Consequently, 
\[
\exp_p(v) = (\exp'_{p'}(v'),\exp''_{p''}(v'')).
\]
Assume that in $(M',g')$, $v_0'$ is conjugate at $p_0$ of fold type, and assume that $v_0''$ is not conjugate at $p_0''$ in $(M'',g'')$. Then 
\[
\d \exp_p(v) = \mbox{diag} (\d \exp'_{p'}(v'),\d\exp_{p''}(v'')).
\]
The kernel of $\d\exp_p(v)$ then consists of $N_p(v)=N_{p'}(v')\times 0$. Next, $S(p) = S(p')\times T_{p''}M''$, and $\Sigma(p) =\Sigma'(p')\times M''$.  Then $N_p(v_0)$ is transversal to $S(p)$ at $v=v_0$, therefore $(v',v'')$ is a fold conjugate vector for $v'\in S'(p)$ close to $v_0$ and for any $v''$. Then the left projection $\pi_{\rm L}$ of the Lagrangian $\mathcal{N}^*\Sigma$ consists of $(p,\xi)$ with $(p',\xi')\in \pi_{\rm L}(\Sigma')$ and $\xi''=0$. Thus the rank drops at least by $n''=\dim(M'')$. We get the same conclusion for $\pi_{\rm R}(\mathcal{N}^*\Sigma)$. Therefore, $\mathcal{N}^*\Sigma$ is not a canonical graph in this case. 

Let $n'=\dim(M')=2$. Then the canonical relation in  $(M',g')$ is a canonical graph, and we get that  $\pi_{\rm L,R}(\mathcal{N}^*\Sigma)$ have rank $2n'+n''=4+n''$ instead of the maximal possible $2n=4+2n''$; i.e., the loss is exactly $n''$.

Assume now that $n'=2$, $n''=1$, and the metric in $M$ is given by 
\[
\sum_{\alpha,\beta=1}^2g_{\alpha\beta}(x^1,x^2)\d x^\alpha \d x^\beta + (\d x^3)^2.
\]
Assume also that in $M'$, we have a fold conjugate vector $v_0=(0,1)$ at $x^1=x^2=0$. Then all possible conormals to the conjugate loci at $(0,0)$ corresponding to small perturbations of $v_0$ will lie in the plane $v^3=0$. This is an example where Corollary~\ref{cor_2} can be applied. We can recover singularities of the kind $\xi = (\xi_1,\xi_2,\xi_3)$ at $p_0=(0,0,0)$ with $\xi_3\not=0$ and $(\xi_1,\xi_2)$ in a conic neighborhood of  $(1,0)$. The ones with $\xi_3=0$ are the problematic ones.

\end{document}